\documentclass{eajam}
\renewcommand\thisyear {200x}

\renewcommand\thisvolume{xx}

\usepackage{charter}
\usepackage[charter]{mathdesign}
\usepackage{booktabs}
\usepackage{graphicx,wrapfig}
\usepackage{algpseudocode,algorithm,algorithmicx}
\usepackage{amssymb,amsmath,epsfig}
\usepackage{epstopdf}
\usepackage{stmaryrd}
\graphicspath{figures/}

\def\bs{\boldsymbol}
\def\il{\interleave}

\newcommand{\zd}{\,\mathrm{d}}
\newcommand{\diff}{\triangledown_{\tau}}

\newcommand{\abs}[1]{\left|#1\right|}
\newcommand{\absb}[1]{\big|#1\big|}
\newcommand{\absB}[1]{\Big|#1\Big|}

\newcommand{\brab}[1]{\big(#1\big)}
\newcommand{\braB}[1]{\Big(#1\Big)}

\newcommand{\brat}[1]{(#1)}
\newcommand{\kbra}[1]{\left[#1\right]}

\newcommand{\kbrat}[1]{[#1]}
\newcommand{\myinner}[1]{\left\langle#1\right\rangle}

\newcommand{\mynorm}[1]{\left\|#1\right\|}
\newcommand{\mynormb}[1]{\big\|#1\big\|}

\newcommand{\h}[1]{\mathbf{#1}}

\begin{document}

\markboth{X. Zhao, H. F. Zhang and H. Sun}{Variable-step BDF2 for MBE with slope selection}
\title{Error analysis of the implicit variable-step BDF2 method for the molecular beam epitaxial model with slope selection}


\author[Xuan Zhao, Haifeng Zhang and Hong Sun]{Xuan Zhao\affil{1}\comma\corrauth , Haifeng Zhang\affil{1} and Hong Sun\affil{1,2} }
\address{\affilnum{1}\ School of
			Mathematics, Southeast University, Nanjing 210096, P. R. China\\
\affilnum{2}\ Department of Mathematics and Physics, Nanjing Institute of Technology, Nanjing 211167}
%
%
\emails{{\tt xuanzhao11@seu.edu.cn} (X. Zhao), {\tt 220211744@seu.edu.cn} (H. F. Zhang), {\tt sunhongzhal@126.com} (H. Sun)}
%
\begin{abstract}
We derive unconditionally stable and convergent variable-step BDF2 scheme for solving the MBE model with slope selection. The discrete orthogonal convolution kernels of the variable-step BDF2 method is commonly utilized recently for solving the phase field models. In this paper, we further prove some new inequalities, concerning the vector forms, for the kernels especially dealing with the nonlinear terms in the slope selection model. The convergence rate of the fully discrete scheme is proved to be two both in time and space in $L^2$ norm under the setting of the variable time steps.  Energy dissipation law is proved rigorously with a modified energy by adding a small term to the discrete version of the original free energy functional. Two numerical examples including an adaptive time-stepping strategy are given to verify the convergence rate and the energy dissipation law.
\end{abstract}



\maketitle



\vskip5mm
	\section{Introduction}
\setcounter{equation}{0}

Over the past decades, the dynamics of molecular beam epitaxy (MBE) model attracted broad interest from the fields of chemistry, material science, mathematics and etc.  The epitaxial growth process offers a controllable method to obtain lateral heterojunction, with an atomically sharp interface, for some attractive materials in making smaller transistors\cite{Li+Science+2015,Wang+Nat+2019}. Atomistic models, continuum models and hybrid models, from various scales, are applied to study the evolution of the surface morphology during epitaxial growth. MBE is the most widely used technique for growing thin epitaxial layers of semiconductor crystals and metallic materials\cite{2018+Jenichen+PRM}. In addition, Nair et al.\cite{2018+Nair+AM} introduced the growth of superconducting $Sr_2RuO_4$ thin films by MBE on (110)$NdGaO_3$ substrates with transition temperatures of up to 1.8 K.

In this paper, we consider the MBE model with slope selection in the two-dimensional domain $\Omega=(0,L)^2\subset \mathbb{R}^2$. Let $u(\mathbf{x},t)$ be the epitaxy surface height with space variable $\mathbf{x}\in \Omega$ and time variable $t\ge 0,$ the height evolution equation \cite{Moldovan+2000+PRE} is expressed as follows
\begin{align}\label{cont: MBE model}
	u_t+\delta\Delta^2u-\nabla\cdot f(\nabla u)=0,~~\mathbf{x}\in\Omega,~0<t\le T,
\end{align}
subjected to the periodic boundary conditions and the initial data $u(\mathbf{x},0)=\varphi_0(\mathbf{x})$. Here, $\delta>0$ is the constant that represents the width of the rounded corners on the otherwise faceted crystalline thin films. The vector $f$ is the nonlinear bulk force, defined by
\begin{align}\label{cont: nonlinear force}
	f(\mathbf{v})=\brat{\abs{\mathbf{v}}^2-1}\mathbf{v}.
\end{align}
When $t\rightarrow\infty,$ one obtains $\left|\nabla\phi\right|\rightarrow1,$ that is why it is called the model with slope selection. There is also a counterpart model, in which $f(\mathbf{v})=-\mathbf{v}/(1+\abs{\mathbf{v}}^2)$, called MBE model without slope selection due to that during the coarsening process $\left|\nabla\phi\right|$ does not converge to a constant.
For any $u\in H^1(\Omega),$ define the energy function by
 $$E(t)=\int_0^t\|u_t(\cdot,\cdot,s)\|^2 ds+\frac{\delta}{2}\|\Delta u(\cdot,\cdot,t)\|^2+\frac{1}{4}\iint\limits_{\Omega}(|\nabla u(x,y,t)|^2-1)^2 dxdy.$$
The following energy dissipation law holds
$$\frac{dE(t)}{dt}=0,~~t>0.$$

The model \eqref{cont: MBE model} has been applied to modeling interfacial coarsening dynamics in epitaxial growth with slope selection, where the fourth-order term models surface diffusion, and the nonlinear second-order term models the well-known Ehrlich-Schowoebel effect, which consequently leads to the formation of mounds and pyramids on the growing surface. Gyure et al. \cite{1998+Gyure+PRL} conducted an experiment to show the unstable growth of thin films on rough surfaces. The MBE of InAs buffer layers is performed on InAs(001) substrates, in the experiment, which exhibit large-small-large wavelength oscillations as the thickness of buffer layers increasing. This morphological instability in the rough-smooth-rough pattern is fundamentally due to the Ehrlich-Schwoebel effect.

The well-posedness for the growth equation with slope selection for different boundary conditions was studied in King et al.\cite{2000+King}. Li and Liu \cite{2003+Li+EJAM} proved the well-posedness and the solution regularity for the initial-boundary-value thin film epitaxy model. The Galerkin spectral method was applied to solve the numerical solution of the model with or without slope selection. In addition, numerical results showed the decay of energy and roughness at different time stages. Li et al. \cite{2017+Li+JDE} analyzed the gradient flow modeling the epitaxial growth of thin films with slope selection in physical dimensions. The improved local and global well-posedness for solutions with critical regularity were established. Several lower and upper bounds for the gradient were obtained.

Due to the high order derivatives and the nonlinear term, it takes a long time to reach the steady state in the dynamics of the MBE model. As is well known that the linearized schemes can avoid solving large nonlinear systems, whereas, the stabilized term usually needs to be added to the scheme in order to guarantee the stability of the scheme. However, the nonlinear schemes, which cause large computational cost,  are usually stable. How to develop proper temporal discretization for the nonlinear term is a key issue to preserve energy stability at the time-discrete level and balance the computational cost. The splitting schemes are adequate choices for the fast simulation. Moreover, variable time-stepping methods are also proved as the efficient techniques, which are fundamentally difficult in the analysis for the long time simulation.

Existing attentions are given to apply the linearized schemes for solving the MBE models with slope selection. The linearized backward Euler difference scheme and the linearized Crank-Nicolson difference scheme were derived in \cite{2012+Qiao+NMPDE}. Yang et al.\cite{2017+Yang+JCP} developed a first and second order time-stepping scheme based on the Invariant Energy Quadratization method, in which all nonlinear terms were treated semi-explicitly. Besides, there were also schemes that result in linear systems at each time step(cf., e.g., \cite{2018+Chen+ANM,2019+Chen+CCP}). For the MBE models without slope selection, we further refer to the monograph \cite{2018+Li+JSCwithout,2019+Cheng+JSCwithout,2020+Chen+RMSwithout,2020+Chen+EMMNAwithout,2021+Zhang+IJNAMwithout,2021+Li+JSCwithout,2021+Hao+CCPwithout} on the linearized schemes.
The stabilized terms are usually utilized in order to preserve the stability of the linearized schemes. Xu and Tang\cite{2006+Xu+SINUM} constructed linearized schemes added with the stabilized terms, which are consistent with the orders of the time discretizations, and showed that the schemes allow much larger time steps than those of a standard implicit-explicit approach. Li et al.\cite{2016+Li+SIAM} proved the unconditional energy stability for the stabilized semi-implicit time-stepping methods without the Lipschitz assumption on the nonlinearity. Utilizing a regularized term, Chen et al.\cite{2019+Chen+CCP} proposed a fully discrete scheme, which preserves energy-dissipation property, for the MBE model with slope selection. The scaling law for the roughness growing and effective energy decaying are captured in the long time simulations. Other cases concerning the effect of the stabilized terms in solving MBE models with slope selection can be found in \cite{2018+Feng+NMPDE,2020+Wang+JCAM}. We also refer the reader to the references therein \cite{2014+Chen+JSCwithout,2018+Li+JSCwithout,2019+Chen+JSCwithout,2019+Cheng+JSCwithout,2020+Chen+RMSwithout,2020+Chen+EMMNAwithout} on the stabilized terms for the MBE models without slope selection.

Whereas, the nonlinear schemes are also selected for numerically solving MBE models due to its advantages in preserving stability in the long time computations. Chen and Wang\cite{2012+Chen+NM}  presented a semi-implicit nonlinear scheme which combined the mixed finite element method and the backward Euler scheme for the thin film epitaxy problem with slope selection. The mixed formulation only needs to use $C^1$ elements by introducing proper dual variables, which are defined naturally from the nonlinear term in the equation. Feng et al.\cite{2018+Feng+NMPDE} studied an implicit nonlinear finite difference scheme using two-step backward differentiation formula (BDF2) method with constant coefficient stabilizing terms for the epitaxial thin film equation with slope selection. The efficient preconditioned steepest descent and the preconditioned nonlinear conjugate gradient algorithms were applied to solve the corresponding nonlinear system. An energy stable, nonlinear mixed finite element scheme was proposed and analyzed for the thin film epitaxial growth model with slope selection\cite{2020+Wang+JCAM}. An optimal convergence rate was obtained with the help of some auxiliary techniques over triangular elements. Additional theoretical frameworks for nonlinear schemes were described in \cite{2015+Qiao+SINUM,2020+Luo+AML}. Furthermore, for the nonlinear schemes for the MBE model without slope selection, we refer to \cite{2014+Chen+JSCwithout,2015+Qiao+MCwithout,2022+Kang+JSCwithout} and the references therein.

One approach to achieve fast simulation appropriately in the presence of the nonlinear terms in MBE model is the splitting method. Cheng et al.\cite{2015+Cheng+JCP} introduced fast explicit operator splitting methods for both one- and two-dimensional nonlinear diffusion equations for thin film epitaxy with slope selection. A fast explicit operator splitting method, which splits the original problem into linear and nonlinear subproblems, was proposed for the epitaxial growth model with slope selection\cite{2017+Li+SINUM}. The convergence rate of the algorithm in discrete $L^2$ norm was analyzed theoretically. Lee et al.\cite{2017+Lee+JSC} developed an operator splitting Fourier spectral method, which alleviates restriction on the time steps, for epitaxial thin film growth with and without slope selection. Different forms about the splitting method for MBE model with slope selection were mentioned in \cite{2010+Wang+DCD,2012+Shen+SINUM,2015+Feng+SINUM,2017+Qiao+IJNAMwith,2019+Cheng+JSC}. As a supplement, we refer to \cite{2012+Chen+JSCwithout,2014+Chen+JSCwithout,2015+Xia+JCPwithout,2018+Ju+MCwithout,2021+Shin+ANMwithout} for the splitting method in the MBE models without slope selection.

An efficient approach for avoiding large computational cost in the long time simulation is the adaptive technique. Qiao et al.\cite{2011+Qiao+SISC} applied time adaptivity strategies for some unconditionally energy stable finite difference schemes and showed that the steady-state solutions and the dynamical changes of the solution are resolved accurately and efficiently. Luo et al.\cite{2020+Luo+AML} developed two types of adaptive time-stepping methods in which equidistribution of the physical quantities in time direction was taken to control the simulation error. Liao et al.\cite{2021+Liao+SCMwithout} introduced the BDF2 scheme with variable steps for the MBE model without slope selection, the $L^2$ norm stability and rigorous error estimates of which were established under an improved step-ratio constraint.
A detailed survey on all related literature would exceed the scope of this paper. Therefore, we confine ourselves to the papers mentioned above and the references therein.

Over the last 20 years, variable contributions have delineated the numerical computations of MBE models. In the current work, we focus on the analysis of the variable-step BDF2 scheme for the model with slope selection. We have followed the analysis of variable-step BDF2 scheme for Cahn-Hilliard model in \cite{2022-Ji+JSC}. In particular, for the kernels and the time-step ratios derived in \cite{2022-Ji+JSC}, we proved Lemma \ref{lem:inequality of the kernel}, Lemma \ref{lem: DOC property2} and Lemma \ref{nonlinear-embedding inequality} concerning the vector forms dealing with the nonlinear terms, for analyzing convergence properties and deriving error estimates of the proposed scheme. The unique solvability and the energy stability of the scheme is demonstrated by virtue of the properties of the convolution coefficients under the same mild time-step ratios restriction.

The rest of the paper is structured as follows. In section 2, we establish an implicit variable-step BDF2 scheme and introduce some preliminary lemmas that facilitate the unique solvability and the energy dissipation law of the difference scheme. In section 3, we introduce some fundamental properties and several discrete convolution inequalities with respect to the discrete orthogonal convolution (DOC) kernels which help demonstrating the error estimate of the proposed scheme. We perform and discuss typical numerical examples in section 4 to verify the theoretical results.

\section{Discrete energy dissipation law} 
\setcounter{equation}{0}

In this section, we investigate the unique solvability of the difference scheme (\ref{eq: fully BDF2 implicit scheme}) based on the Brouwer fixed-point theorem. By virtue of the properties of the convolution coefficients, the energy stability of the scheme is demonstrated. We start with the introduction of some notations.

For the spatial direction discretization, let $M$ be a positive integer, $\Omega=(0,L)^2$ and $x_i=ih$, $y_j=jh$ with the spatial lengths $h_x=h_y=h:=L/M.$
Set the discrete spatial grid
$\Omega_h:=\big\{(x_i, y_j)~|~1\le i, j\le M\big\}$
and $\bar{\Omega}_h:=\big\{(x_i, y_j)~|~0\le i, j\le M\big\}.$
Consider the $L$-periodic function space
$$\mathbb{V}_h:=\big\{v_h=v(\mathrm{x}_h)~|~\mathrm{x}_h=(x_i, y_j)\in \bar{\Omega}_h
\;\text{and $v_h$ is $L$-periodic in each direction}\big\}.$$
Given a grid function $v\in \mathbb{V}_h,$ introduce the following notations
$\delta_xv_{i+\frac{1}{2},j}=(v_{i+1,j}-v_{ij})/h$, $\Delta_xv_{ij}=(v_{i+1,j}-v_{i-1,j})/(2h),$ and $\delta^2_xv_{ij}=(\delta_xv_{i+\frac{1}{2},j}-\delta_xv_{i-\frac{1}{2},j})/h.$
The discrete notations $\delta_yv_{i,j+\frac{1}{2}}$, $\Delta_yu_{ij}$ and $\delta^2_yv_{ij}$ can be defined similarly.
Also, we define the discrete Laplacian operator  $\Delta_hv_{ij}=\delta^2_xv_{ij}+\delta^2_yv_{ij}$ and
the discrete gradient vector $\nabla_hv_{ij}=(\Delta_xv_{ij},~\Delta_yv_{ij})^T.$

For any grid functions $v,w\in\mathbb{V}_h,$ define the inner product
$\myinner{v,w}:=h^2\sum\limits_{\mathrm{x}_h\in\Omega_h}v_hw_h,$  the associated $L^2$ norm
$\mynorm{v}:=\sqrt{\myinner{v,v}}$,
and the discrete $L^q$ norm
$\mynorm{v}_{l^q}:=\sqrt[q]{h^2\sum_{\mathrm{x}_h\in\Omega_h}\abs{v_h}^q}.$
The discrete seminorms $\|\nabla_hv\|_{l^q}$ and $\mynorm{\Delta_hv}_{l^q}$ are defined similarly.
In addition, $\|\delta_xv\|$ and $ \|\delta_yv\|$ are written as $$\|\delta_xv\|=\sqrt{h^2\sum_{\mathrm{x}_h\in\Omega_h}
\big(\delta_xv_{i-\frac{1}{2}, j}\big)^2},\; \|\delta_yv\|=\sqrt{h^2\sum_{\mathrm{x}_h\in\Omega_h}
\big(\delta_yv_{i, j-\frac{1}{2}}\big)^2}.$$  The discrete $H^1$-seminorm is defined by $|v|_1=\sqrt{\|\delta_xv\|^2+\|\delta_yv\|^2}.$ Furthermore,
the discrete Green's formula with periodic boundary conditions yield
$\myinner{\Delta_h^2 v,w}=\myinner{\Delta_hv,\Delta_hw}$ and $-\myinner{\nabla_h\cdot\nabla_hv, w}=\myinner{\nabla_hv, \nabla_hw}$.
In the subsequent analysis, we need the commonly used discrete Sobolev embedding inequality
 \begin{align}
 \mynorm{\nabla_hv}^2\le \mynorm{\Delta_hv}\cdot\mynorm{v}.\label{embedding inequality}
 \end{align}

For the time discretization, take time levels $0=t_0<t_1<t_2<\cdots<t_N=T$ with the time-step $\tau_n=t_n-t_{n-1}$ for $1\le n\le N.$
Let the adjacent time-step ratios $r_n:=\tau_n/\tau_{n-1}$ for $2\le n\le N.$
For any grid function $v^n=v(t_n)$, we denote $\diff v^n:=v^n-v^{n-1}$ and $\partial_{\tau}v^n:=\diff v^n/\tau_n$.
The well-known variable-step BDF2 formula reads
$$D_2v^n:=\frac{1+2r_n}{\tau_n(1+r_n)}\diff v^n-\frac{r_n^2}{\tau_n(1+r_n)}\diff v^{n-1}\quad\text{for $n\geq 2$}.$$
Always, one needs a starting scheme to compute the first-level solution $v^1$
since the two-step BDF2 formula needs two starting values.
To improve the temporal accuracy at the time $t=t_1$, we adopt a second-order accurate approach
$$D_2v^1:=\frac{2}{\tau_1}\kbra{v^1-\brab{v(t_0)+\frac{\tau_1}{2}v_t(t_0)}}$$
using the fact that $\frac1{2}\kbrat{v_t(t_1)+v_t(t_0)}=\frac1{\tau_1}\brat{v^1-v(t_0)}+O(\tau_1^2)$.

We give the implicit variable-step BDF2 scheme for the MBE problem \eqref{cont: MBE model} as
\begin{align}\label{eq: fully BDF2 implicit scheme}
	D_2u_h^n+\delta\Delta_h^2u_h^n-\nabla_h\cdot f(\nabla_h u_h^n)=0\quad\text{for $\mathrm{x}_h\in\Omega_h$, $1\le n\le N$}
\end{align}
with the initial data $u_h^0=\varphi_0(\mathrm{x}_h)-\frac{\tau_1}{2}\varphi_1(\mathrm{x}_h)\ \text{for $\mathrm{x}_h\in\bar{\Omega}_h$},$
where $\varphi_1:=\nabla\cdot f(\nabla \varphi_0)-\delta\Delta^2\varphi_0$ for the smooth data $\varphi_0\in H^4(\Omega)$.
The spatial operators are approximated by the finite difference method. 
We start our analysis by viewing the above BDF2 formula as a discrete convolution
summation $D_2v^n:=\sum_{k=1}^nb_{n-k}^{(n)}\diff v^k\ \text{for $n\ge1$},$
where the discrete convolution kernels $b_{n-k}^{(n)}$ are
defined by $b^{(1)}_0:=2/\tau_1$, and when $n\ge 2$,
\begin{align}\label{eq: BDF2-kernels}
	b_{0}^{(n)}:=\frac{1+2r_n}{\tau_n(1+r_n)},\quad
	b_{1}^{(n)}:=-\frac{r_n^2}{\tau_n(1+r_n)}\quad \text{and} \quad
	b_{j}^{(n)}:=0\quad \mathrm{for}\quad 2\le j\le n-1.
\end{align}
The variable-step BDF2 time-stepping was considered recently in \cite{LiaoZhang:2019,2022-Ji+JSC} from
a new point of view by making the virtue of the positive definiteness of BDF2 convolution kernels $b_{n-k}^{(n)}$.
A concise $L^2$ norm stability and convergence theory of variable-step BDF2 scheme has been established for the linear diffusion equations
provided that the adjacent time-step ratios $r_k\le r_s< 4.864\ \text{for}\ 2\le k\le N$.
The discrete tool as a counterpart is the so-called DOC kernels,
given by
\begin{align}\label{eq: DOC-Kernels}
	\theta_{0}^{(n)}:=\frac{1}{b_{0}^{(n)}}
	\quad \mathrm{and} \quad
	\theta_{n-k}^{(n)}:=-\frac{1}{b_{0}^{(k)}}
	\sum_{j=k+1}^n\theta_{n-j}^{(n)}b_{j-k}^{(j)}
	\quad \text{for $1\le k\le n-1$},
\end{align}
deduced by the following  discrete orthogonal identity
\begin{align}\label{eq: orthogonal identity}
	\sum_{j=k}^n\theta_{n-j}^{(n)}b_{j-k}^{(j)}\equiv \delta_{nk}\quad\text{for $1\le k\le n$,}
\end{align}
where $\delta_{nk}$ is the Kronecker delta symbol. By exchanging the summation order
and using the identity \eqref{eq: orthogonal identity}, it is not difficult to check that
\begin{align}\label{eq: orthogonal equality for BDF2}
	\sum_{j=1}^n\theta_{n-j}^{(n)}D_2v^j=\diff v^n\quad\text{for $\{v^j\,|\,0\le j\le n\}$.}
\end{align}
This equality \eqref{eq: orthogonal equality for BDF2} will play an important role in the subsequent analysis.
The detailed properties of the DOC kernels $\theta_{n-k}^{(n)}$ are referred to Lemma \ref{lem: DOC property}.
Lemma \ref{lem:Conv-Kernels-Positive} shows that the BDF2 convolution kernels $b_{n-k}^{(n)}$ are positive definite
provided the adjacent time-step ratios $r_k$ satisfy a sufficient condition $r_k\le r_s$ for $2\le k \le N$.

\subsection{Unique solvability}
To prove the unique solvability, we need the following lemma.

\begin{lemma}\label{lem:nonlinear term}
For any vectors $\bs u$, $\bs v$ and $\bs z:=\bs u-\bs v$, it holds that
\begin{align*}
&(\bs u-\bs v)^T f(\bs u)\geq \frac{1}{4}(|\bs u|^4-|\bs v|^4)-\frac{1}{2}(|\bs u|^2|-|\bs v|^2)-\frac{1}{2}|\bs u- \bs v|^2,\\
&\bs z^T\kbra{f\brat{\bs u}-f\brat{\bs v}}
\ge\,\frac12|\bs v|^2|\bs z|^2+\frac12\brat{\bs u^T\bs z}^2-|\bs z|^2.\quad\quad\quad\quad\quad
 \end{align*}
\end{lemma}
\begin{proof} We observe the fact that $u^T(\bs u-\bs v)=\frac{1}{2}(|\bs u|^2-|\bs v|^2+|\bs u-\bs v|^2),$ it follows by omitting the nonnegative term $|\bs u|^2|\bs u-\bs v|^2$ and the mean value inequality
\begin{align*}
(\bs u-\bs v)^T f(\bs u)=&(|\bs u|^2-1)\bs u^T(\bs u-\bs v)\\
=&\frac{1}{2}(|\bs u|^2-1)(|\bs u|^2-|\bs v|^2+|\bs u-\bs v|^2)\\
\geq&\frac{1}{4}(|\bs u|^4-|\bs v|^4)-\frac{1}{2}(|\bs u|^2-|\bs v|^2)-\frac{1}{2}|\bs {u-v}|^2.
\end{align*}
It follows from Young's inequality that
\begin{align*}
\bs z^T[ f(\bs u)- f(\bs v)]=&\,\big[|\bs u|^2\bs u-|\bs v|^2\bs v-\bs z\big]^T\bs z\\
=&\,\big[(|\bs u|^2-|\bs v|^2)\bs u+|\bs v|^2\bs z\big]^T\bs z-|\bs z|^2\\
=&\,\big[(\bs u^T\bs z+\bs v^T\bs z)\bs u\big]^T\bs z+|\bs v|^2|\bs z|^2-|\bs z|^2\\
=&\,|\bs v|^2|\bs z|^2+(\bs u^T\bs z+\mathbf{v}^T\bs z)\bs u^T\bs z-|\bs z|^2\\
\ge&\,\frac12|\bs v|^2|\bs z|^2+\frac12\brat{\bs u^T\bs z}^2-|\bs z|^2.
\end{align*}
This completes the proof.
\end{proof}


\begin{theorem} Suppose the time-step ratios satisfy $r_k\le r_s$ for $2\le k\le N$ and the time-step size $\tau_n<\frac{4\delta(1+2r_n)}{1+r_n}$ for $1\le n\le N$. The difference scheme (\ref{eq: fully BDF2 implicit scheme}) is uniquely solvable.
\end{theorem}
\begin{proof} The Brouwer fixed-point theorem is applied to show the solvability of the difference scheme (\ref{eq: fully BDF2 implicit scheme}).
For any fixed index $n\ge1$, we construct the  map $\Pi_n:\mathbb{V}_{h}\rightarrow \mathbb{V}_{h}$ as follows
\begin{align}\label{eq: nonlinear map}
\Pi_n(w_h):=b_0^{(n)}w_h-g_h^{n-1}+\delta\Delta_h^2w_h-\nabla_h\cdot f(\nabla_h w_h),~~\mathrm{x}_h\in\bar{\Omega}_h.
\end{align}
where $g_h^{n-1}=b_0^{(n)}u_h^{n-1}-b_1^{(n)}\diff  u_h^{n-1},$ for $n\ge2$ and $g_h^{0}=b_0^{(1)}u_h^{0}$.

Suppose $u^{n-1},~u^{n-2}$ have been determined, taking the inner product of $\Pi_n(w)$ with $w,$ it yields
\begin{align*}
\myinner{\Pi_n(w), w}=b_0^{(n)}\myinner{w, w}+\delta\myinner{\Delta_h^2w, w}
-\myinner{\nabla_h\cdot f(\nabla_h w), w}-\myinner{g^{n-1}, w}.
\end{align*}
Combining the embedding inequality (\ref{embedding inequality}) and Young's inequality, it follows that
\begin{align*}
\myinner{\Pi_n(w), w}\geq&\,b_0^{(n)}\|w\|^2+\delta\|\Delta_hw\|^2+\|\nabla_hw\|^4_{l^4}-\|\nabla_hw\|^2-\|g^{n-1}\|\cdot\|w\|
\\
\geq&\,b_0^{(n)}\|w\|^2+\delta\|\Delta_hw\|^2-\big(\delta\|\Delta_hw\|^2+\frac{1}{4\delta}\|w\|^2\big)-\|g^{n-1}\|\cdot\|w\|\\
=&\,\big(b_0^{(n)}-\frac{1}{4\delta}\big)\|w\|^2-\|g^{n-1}\|\cdot\|w\|.
\end{align*}
When $\tau_n<\frac{4\delta(1+2r_n)}{1+r_n}$ and $\|w\|=4\delta\|g^{n-1}\|/(4\delta b_0^{(n)}-1),$ we arrive at
$$\myinner{\Pi_n (w), w}\geq0.$$
With the help of the Brouwer fixed-point theorem, there exists a $w_h^*$ such that
$\Pi_n(w_h^*)=0$
which implies that the variable-step BDF2 scheme (\ref{eq: fully BDF2 implicit scheme}) is solvable.

Next, we show the uniqueness of the solutions.
Suppose both $w_h$ and $v_h$ are the solutions of the difference scheme (\ref{eq: fully BDF2 implicit scheme}). Denote the difference $\rho_h=w_h-v_h.$
 Then it follows that
\begin{align}
b_0^{(n)}\rho_h+\delta\Delta_h^2\rho_h-\nabla_h\cdot\Big(f(\nabla_h w_h)- f(\nabla_h v_h)\Big)=0.\label{eq: Error nonlinear map}
\end{align}
Taking the inner product of (\ref{eq: Error nonlinear map}) with $\rho$, we have
\begin{align*}
b_0^{(n)}\|\rho\|^2+\delta\|\Delta_h\rho\|^2+\myinner{f(\nabla_hw)-f(\nabla_hv), \nabla_h\rho}=0.
\end{align*}
Making use of Lemma \ref{lem:nonlinear term} and taking $\h u=\nabla_hu,~~\h v=\nabla_hv$ and $\h z=\nabla_h\rho$, it yields
\begin{align*}
b_0^{(n)}\|\rho\|^2+\delta\|\Delta_h\rho\|^2+
\frac{1}{2}h^2\sum_{\mathbf{x}_h\in\Omega_h}\big(|\nabla_h v_h|^2|\nabla_h\rho_h|^2+(\nabla_h u_h\cdot\nabla_h\rho_h)^2\big)\le&\|\nabla_h \rho\|^2.
\end{align*}
Noticing that the third term on the left hand side is nonnegative, with the help of the embedding inequality \eqref{embedding inequality}, we have
\begin{align*}
b_0^{(n)}\|\rho\|^2+\delta\|\Delta_h\rho\|^2
\le\delta\|\Delta_h\rho\|^2+\frac{1}{4\delta}\|\rho\|^2.
\end{align*}
When $\tau_n<\frac{4\delta(1+2r_n)}{1+r_n},$ the above inequality implies that $\|\rho\|=0.$
This completes the proof.\end{proof}

\subsection{Energy dissipation law}
Now we present the energy stability of the scheme (\ref{eq: fully BDF2 implicit scheme}).
The following lemma shows the convolution kernels $b_{n-k}^{(n)}$ are positive definite if the time-step ratios satisfy $r_k\le r_s$ for $2\le k\le N$.

\begin{lemma}{\rm\cite{2022-Ji+JSC}}\label{lem:Conv-Kernels-Positive}
Let the time-step ratios satisfy $r_k\le r_s$ for $2\le k\le N$, for any real sequence $\{w^k\}_{k=1}^n$ with n entries, it holds that
\begin{align*}
2w_k\sum_{j=1}^kb_{k-j}^{(k)}w_j
&\ge\frac{r_{k+1}^{\frac{3}{2}}}{1+r_{k+1}}\frac{w_k^2}{\tau_k}
-\frac{r_k^{\frac{3}{2}}}{1+r_k}\frac{w_{k-1}^2}{\tau_{k-1}}
+R_L(r_k,r_{k+1})\frac{w_k^2}{\tau_k} ~~\text{for}~~ k\ge2,
\end{align*}
where $R_L(z,s)=\frac{2+4z-z^{\frac{3}{2}}}{1+z}-\frac{s^{\frac{3}{2}}}{1+s},~0<z,s<r_s.$
Thus the discrete convolution kernels $b_{n-k}^{(n)}$ are positive definite
\[
\sum_{k=1}^n w_k \sum_{j=1}^k b_{k-j}^{(k)}w_j\ge \frac12\sum_{k=1}^nR_L(r_k,r_{k+1})\frac{w_k^2}{\tau_k}\quad\text{for $n\ge 1$}.
\]
\end{lemma}
We define the discrete energy $E^n=\frac{\delta}{2}\mynormb{\Delta_hu^n}^2+\frac{1}{4}\mynormb{|\nabla_hu^n|^2-1}^2,\ \text{for}~n\ge0.$
Furthermore, the modified discrete energy is defined by $\mathcal{E}^n=E^n+\frac{r_{n+1}^{\frac32}}{2(1+r_{n+1})\tau_n}\mynormb{\diff  u^n}^2,$ for $n\ge 1$, with $\mathcal{E}^0=E^0.$
\begin{theorem}\label{Energy dissipation law}
If the time-step ratios satisfy $r_k\le r_s$ for $2\le k\le N$ and the time-step size
\begin{align}
\tau_n< 4\delta \min\Big\{R_L(r_n,r_{n+1}),\frac{2+r_2}{1+r_2}\Big\},\label{Restriction-Time-Step}
\end{align}
then the solution of the variable-step BDF2 scheme \eqref{eq: fully BDF2 implicit scheme} satisfies
$$E^n\le \mathcal{E}^n\le \mathcal{E}^{n-1}\le E^0,~~1\le n\le N.$$
\end{theorem}
\begin{proof} We start our proof by taking the inner product of (\ref{eq: fully BDF2 implicit scheme}) with $\diff  u^n,$ it yields
\begin{align}\label{Energy-Law-Inner}
\myinner{D_2u^n, \diff  u^n}+\delta\myinner{\Delta_h^2u^n, \diff  u^n}-\myinner{\nabla_h\cdot f(\nabla_h u^n), \diff  u^n}=0.
\end{align}
For $n\geq 2,$ taking $w_n=\diff u^{n}$ in the first inequality of Lemma \ref{lem:Conv-Kernels-Positive}, we obtain an estimate of the first term on the left hand in (\ref{Energy-Law-Inner}),
\begin{align*}
\myinner{D_2u^n, \diff  u^n}\geq&\frac{r_{n+1}^{\frac32}}{2(1+r_{n+1})\tau_n}\|\diff  u^n\|^2
-\frac{r_{n}^{\frac32}}{2(1+r_{n})\tau_{n-1}}\|\diff  u^{n-1}\|^2\\
&+\frac{1}{2\tau_n}R_L(r_n,r_{n+1})\|\diff  u^n\|^2.
\end{align*}
By virtue of the summation by parts and the equality $2a(a-b)=a^2-b^2+(a-b)^2,$ we have the diffusion term rewritten as
\begin{align*}
 \myinner{\Delta_h^2u^n,\diff  u^n}
 =\frac{1}{2}(\|\Delta_hu^n\|^2-\|\Delta_hu^{n-1}\|^2+\|\Delta_h(\diff  u^n)\|^2).
 \end{align*}
An application of Lemma \ref{lem:nonlinear term} with $\h u=\nabla_h u^n$ and $\h v=\nabla_h u^{n-1}$ gives an lower bound of the nonlinear part
 \begin{align*}
&-\myinner{\nabla_h\cdot f(\nabla_h u^n), \diff  u^n}=\myinner{f(\nabla_h u^n), \diff  (\nabla_h u^n)}\\
\geq& \frac{1}{4}(\|\nabla_hu^n\|_{l^4}^4-\|\nabla_hu^{n-1}\|_{l^4}^4)-\frac{1}{2}(\|\nabla_h u^n\|^2-\|\nabla_h u^{n-1}\|^2)-\frac{1}{2}\|\nabla_h (\diff  u^n)\|^2\\
\ge&\frac{1}{4}\||\nabla_hu^n|^2-1\|^2-\frac{1}{4}\||\nabla_hu^{n-1}|^2-1\|^2-\frac{1}{2}\|\nabla_h (\diff  u^n)\|^2.
\end{align*}

Substituting the above treatments into (\ref{Energy-Law-Inner}), the modified discrete energy $\mathcal{E}^n$ is given in the whole inequality
\begin{align*}
&\mathcal{E}^n+\frac{1}{2\tau_n}R_L(r_n,r_{n+1})\|\diff  u^n\|^2
+\frac{\delta}{2}\|\Delta_h(\diff  u^n)\|^2\nonumber\\
\le &\mathcal{E}^{n-1}+\frac{1}{2}\|\nabla_h (\diff  u^n)\|^2\\
\le &\mathcal{E}^{n-1}+\frac{\delta}{2}\|\Delta_h(\diff  u^n)\|^2+\frac{1}{8\delta}\|\diff  u^n\|^2,
~~2\le n\le N,
\end{align*}
in which the embedding inequality (\ref{embedding inequality}) is utilized.
Afterwards it follows from (\ref{Restriction-Time-Step}) that
$\mathcal{E}^n\le \mathcal{E}^{n-1},~2\le n\le N.$
For $n=1,$ by Young's inequality, it yields
\begin{align*}
\myinner{D_2u^1, \diff  u^1}
&\geq \frac{r_{2}}{2(1+r_2)\tau_1}\|\diff u^1\|^2+\frac{2+r_2}{2(1+r_2)\tau_1}\|\diff u^1\|^2.
\end{align*}
Then, when $\tau_1< \frac{4\delta (2+r_2)}{1+r_2},$ we have
$ \mathcal{E}^1\le \mathcal{E}^0.$
Thus, it is obvious that $E^n\le \mathcal{E}^n\le\mathcal{E}^0=E^0.$
\end{proof}

\begin{lemma}\label{bound-numerical-solution}
If the time-step ratios satisfy $r_k\le r_s$ for $2\le k\le N$ and the condition (\ref{Restriction-Time-Step}) holds, the numerical solution of the variable-step BDF2 scheme (\ref{eq: fully BDF2 implicit scheme}) satisfies
$$\max\{\|\nabla_hu^n\|, \|\nabla_hu^n\|_{l^4}, \|\Delta_hu^n\|\}\le C_0,$$
where $C_0$ is a constant, which is independent of the spatial lengths $h$ and the time steps $\tau_n.$
\end{lemma}
\begin{proof} Using the fact that $a^4\ge 4a^2-4$  and applying Theorem \ref{Energy dissipation law}, it follows that
\begin{align*}
4E^0&\ge4E^n=2\delta\|\Delta_hu^n\|^2+\|\nabla_hu^n\|_{l^4}^4-2\|\nabla_hu^n\|^2+|\Omega_h|\\
&\ge2\delta\|\Delta_hu^n\|^2+\frac{1}{4}\|\nabla_hu^n\|_{l^4}^4+\frac{3}{4}(4\|\nabla_hu^n\|^2-4|\Omega_h|)-2\|\nabla_hu^n\|^2+|\Omega_h|\\
&\ge2\delta\|\Delta_hu^n\|^2+\frac{1}{4}\|\nabla_hu^n\|_{l^4}^4+\|\nabla_hu^n\|^2-2|\Omega_h|.
\end{align*}
Taking $K_0=\min\{2\delta, 1/4\}$, we get
\begin{align*}
\|\Delta_hu^n\|^2+\|\nabla_hu^n\|_{l^4}^4+\|\nabla_hu^n\|^2\le(4E^0+2|\Omega_h|)/K_0=:K_1.
\end{align*}
Then, the desired estimate is obtained with $C_0=\max\{\sqrt{K_1}, \sqrt[4]{K_1}\}.$
\end{proof}

\section{Convergence analysis}
\setcounter{equation}{0}
In this section, we derive the error estimate of the implicit variable-step BDF2 scheme (\ref{eq: fully BDF2 implicit scheme}).
We begin with some fundamental properties and several discrete convolution inequalities with respect to the DOC kernels. For convenience, we firstly introduce the denotation $\sum\limits_{k,l}^{n,k}:=\sum\limits_{k=1}^n
\sum\limits_{l=1}^k.$

\begin{lemma}{\rm \cite{LiaoZhang:2019}}\label{lem: DOC property}
	If the discrete convolution kernels  $b^{(n)}_{n-k}$ defined in
	\eqref{eq: BDF2-kernels} are positive definite,
then the DOC kernels $\theta_{n-l}^{(n)}$
	defined in \eqref{eq: DOC-Kernels} satisfy:
\begin{itemize}
  \item[(I)] The discrete kernels $\theta_{n-l}^{(n)}$ are positive definite;
  \item[(II)] $\displaystyle \sum_{l=1}^{n}\theta_{n-l}^{(n)}\le\tau_n$ such that
  $\displaystyle \sum_{k,l}^{n,k}\theta_{k-l}^{(k)}\le t_n$ for $n\ge1$.
\end{itemize}
\end{lemma}

The following three lemmas are the key to proving the convergence of the BDF2 scheme (\ref{eq: fully BDF2 implicit scheme}) for dealing with the nonlinear term. We describe the lemmas in detail below but put their proofs to the Appendix for brief.

\begin{lemma}\label{lem:inequality of the kernel}
Assuming that the time-step ratios satisfy $r_k\le r_s$ for $2\le k\le N$, for any vector sequence $\{\bs v^k\}_{k=1}^n\in \mathbb{R}^2$, the following inequality holds
\begin{align*}
\frac{\mathfrak{m}_1}{2\mathfrak{m}_2}\sum_{k=1}^n\tau_k(\bs v^k)^T\bs v^k\le\sum_{k,l}^{n,k}\theta_{k-l}^{(k)}(\bs v^l)^T\bs v^k\le \frac{\mathfrak{m}_3}{2}\sum_{k=1}^n\tau_k(\bs v^k)^T\bs v^k,
\end{align*}
where $\mathfrak{m}_1, \mathfrak{m}_2, \mathfrak{m}_3$ are positive constants.
\end{lemma}

\begin{lemma}\label{lem: DOC property2}
If the time-step ratios satisfy $r_k\le r_s$ for $2\le k\le N$, for any vector sequences $\{\bs v^k, \bs w^k\}_{k=1}^n,$ where $\bs v^k, \bs w^k\in \mathbb{R}^2$, it holds that
\begin{align*}
\sum_{k,l}^{n,k}\theta_{k-l}^{(k)}(\bs{v}^l)^T \bs{w}^k\le \epsilon\sum_{k=1}^{n}\tau_k(\bs{v}^k)^T\bs{v}^k+
\frac{\mathfrak{m}_{3}}{4\mathfrak{m}_1\epsilon}\sum_{k=1}^{n}\tau_k(\bs{w}^k)^T\bs{w}^k ~~for~ \forall \epsilon>0.
\end{align*}
\end{lemma}

\begin{lemma}\label{nonlinear-embedding inequality}
Assume that the time-step ratios satisfy $r_k\le r_s$ for $2\le k\le N$, consider the grid function $u^k$ and any vector sequences $\{\bs{z}^k, \bs{w}^k\}_{k=1}^n,$
where $\bs{z}^k, \bs{w}^k\in \mathbb{R}^2, $
and there exists a constant $C_u$ such that $$C_u=\max_{1\le j\le n}\|u^j\|_{l^3}.$$
Then it holds that
\begin{align*}
\sum_{k,l}^{n,k}\theta_{k-l}^{(k)}\langle u^l\bs{z}^l, \bs{w}^k\rangle
\le \varepsilon\sum_{k=1}^{n}\tau_k(\|\nabla_h\bs{z}^k\|^2+\|\bs z^k\|^2)
+\frac{C_\Omega C_u^2\mathfrak{m}_3}{4\mathfrak{m}_1\varepsilon}\sum_{k=1}^{n}\tau_k\|\bs{w}^k\|^2.
\end{align*}
\end{lemma}

The following embedding inequalities on $\|\nabla_h v\|_{l^4}$ and $\|\nabla_h v\|_{l^6}$ are used to control the norms in the convergence analysis of the BDF2 scheme (\ref{eq: fully BDF2 implicit scheme}).

\begin{lemma}\label{gradient-bound}
For the grid functions $v_h\in \mathbb{V}_h, ~\mathrm{x}_h\in\Omega_h,$ it holds that
\begin{align}
\|\nabla_hv\|_{l^4}\le K_1 \|\nabla_hv\|^\frac{1}{2}\Big(2\|\Delta_hv|^2+\frac{1}{L^2}\|\nabla_hv\|^2\Big)^\frac{1}{4}\label{l_4embedding inequality}
\end{align}
and
\begin{align}
\|\nabla_hv\|_{l^6}\le K_2\|\nabla_hv\|^{\frac{1}{3}}\braB{16\|\Delta_hv\|^2+\frac{1}{L^2}\|\nabla_hv\|^2}^{\frac{1}{3}},\label{l_6embedding inequality}
\end{align}
where $K_1$ and $K_2$ are two constants.
\end{lemma}
By virtue of Lemma \ref{bound-numerical-solution} and Lemma \ref{gradient-bound}, there exists a constant $K$ such that
\begin{align}\label{l6boundedness}
\|\nabla_hu^n\|_{l^6}\le K.
\end{align}

Now we present the error behavior of BDF2 time-stepping
with respect to the variation of time-step sizes with the following two lemmas.
\begin{lemma}{\rm\cite{LiaoZhang:2019}}\label{lem:BDF2-Consistency-Error}
For the consistency error $\xi^j=D_2u(t_j)-\partial_tu(t_j)$ at $t=t_j$, Let $P^k$ be a convolutional consistency error, defined by
$P^k:=\sum\limits_{j=1}^k\theta_{k-j}^{(k)}\xi^j.$
If the time-step ratios satisfy $r_k\le r_s$ for $2\le k\le N$, the convolutional consistency error $P^k$ satisfies
\begin{align*}
\sum_{k=1}^n\absb{P^k}
\le&\,3t_n\max_{1\leq j\leq n}\Big(\tau_{j}\int_{t_{j-1}}^{t_j} \absb{u'''(s)} \zd{s}\Big) \quad\text{for $2\le n\le N.$}
\end{align*}
\end{lemma}

\begin{lemma}{\rm\cite{LiaoZhang:2019}}\label{lem: Gronwall inequality}
Let $\lambda\geq0,$ the sequences $\{\xi_j\}_{j=0}^N$ and $\{V_j\}_{j=1}^N$ be nonnegative. If
$$V_n\le \lambda\sum_{j=1}^{n-1}\tau_jV_j+\sum_{j=1}^{n}\xi_j ~~for~ 1\le n\le N,$$
then it holds that
$$V_n\le \exp(\lambda t_{n-1})\sum_{j=1}^{n}\xi_j ~~for~ 1\le n\le N.$$
\end{lemma}

Now, we set about to demonstrate the convergence of the variable-step BDF2 scheme (\ref{eq: fully BDF2 implicit scheme}). Let $C_{1}=\max_{(x,y,t)\in\Omega\times[0,T]}\big\{|u(x,y,t)|,|\nabla u(x,y,t)|,|\Delta u(x,y,t)|\big\}.$
Denoting that $ e_h^n=U_h^n-u_h^n,\ \mathrm{x}_h\in\overline{\Omega}_h,~ 0\le n\le N,$
 we get the error equation as follows
 \begin{align}
 &D_2e_h^n+\delta\Delta_h^2e_h^n-\nabla_h\cdot(f(\nabla_h U_h^n)-f(\nabla_h u_h^n))=\xi_h^n+\eta_h^n,~~\mathrm{x}_h^n\in\Omega_h,~1\le n\le N, \label{error scheme}
 \end{align}
where $\xi_h^n,~\eta_h^n$ denote the local consistency error in time and space.
 \begin{theorem}\label{th3.2}
Suppose the problem (\ref{cont: MBE model}) has a unique smooth solution and $u_h^n\in\mathbb{V}_h$ is the solution of the difference scheme  (\ref{eq: fully BDF2 implicit scheme}). If the time-step ratios satisfy $r_k\le r_s$ for $2\le k\le N$ with the maximum time-step size $\tau\le\tau_0,$ where $\tau_0$ is a constant, the variable-step BDF2 scheme (\ref{eq: fully BDF2 implicit scheme}) is convergent in $L^2$ norm.
$$ \|e^n\|\le C(\tau^2+h^2),~~1\le n\le N.$$
\end{theorem}
\begin{proof} Replacing $n$ by $l$ in (\ref{error scheme}), multiplying both sides of (\ref{error scheme})
 by the DOC kernels $\theta_{k-l}^{(k)}$ and summing $l$ from 1 to $k$, then it yields
\begin{align}
\sum_{l=1}^k\theta_{k-l}^{(k)}D_2e_h^l+\delta\sum_{l=1}^k\theta_{k-l}^{(k)}\Delta_h^2 e_h^l
-\sum_{l=1}^k\theta_{k-l}^{(k)}\nabla_h\cdot\brab{f(\nabla_hU_h^l)-f(\nabla_hu_h^l)}
=\sum\limits_{l=1}^k\theta_{k-l}^{(k)}(\xi_h^l+\eta_h^l).\label{Error-Equation-DOC}
\end{align}
Taking the inner product of (\ref{Error-Equation-DOC}) with $e^k$ and summing $k$ from 1 to $n$, we get
 \begin{align}\label{inner error equality}
 &\sum_{k,l}^{n,k}\langle \theta_{k-l}^{(k)}D_2e^l, e^k\rangle+\delta\sum_{k,l}^{n,k}\theta_{k-l}^{(k)}\langle\Delta_h^2e^l, e^k\rangle+\sum_{k,l}^{n,k}\theta_{k-l}^{(k)}\langle f(\nabla_h U^l)-f(\nabla_h u^l),\nabla_h e^k\rangle\nonumber\\
 =&\sum_{k,l}^{n,k}\theta_{k-l}^{(k)}\langle\xi^l+\eta^l, e^k\rangle.
\end{align}

Noticing equality \eqref{eq: orthogonal equality for BDF2} dealing with the BDF2 discretization, we have the equality $\sum_{l=1}^k\theta_{k-l}^{(k)}D_2e_h^l=\diff  e_h^k.$
Making use of the equality $a(a-b)=\frac{1}{2}[a^2-b^2+(a-b)^2],$ the following result holds for the first term on the left hand side of \eqref{inner error equality}

\begin{align}
\sum_{k=1}^n\langle \nabla_\tau e^k, e^k\rangle=\frac{1}{2}(\|e^n\|^2-\|e^0\|^2)+\frac{1}{2}\sum_{k=1}^n\|\diff e^k\|^2.\label{BDF2-term-inner}
\end{align}

By using the summation by parts and Lemma \ref{lem:inequality of the kernel}, the diffusion term yields
\begin{align}
\sum_{k,l}^{n,k}\theta_{k-l}^{(k)}\langle\Delta_h^2e^l, e^k\rangle=\sum_{k,l}^{n,k}\theta_{k-l}^{(k)}\langle\Delta_he^l,\Delta_he^k\rangle
\ge\frac{\mathfrak{m}_1}{2\mathfrak{m}_2}\sum_{k=1}^n\tau_k\|\Delta_he^k\|^2.\label{equality2}
\end{align}

For the nonlinear term, using the identity
\begin{align}
f(\nabla_hU_h^l)-f(\nabla_hu_h^l)&=|\nabla_hU_h^l|^2\nabla_he_h^l+(|\nabla_hU_h^l|^2-|\nabla_hu_h^l|^2)\nabla_hu_h^l-\nabla_he_h^l\nonumber\\
&=|\nabla_hU_h^l|^2\nabla_he_h^l+(\nabla_hU_h^l\cdot\nabla_he_h^l+\nabla_he_h^l\cdot\nabla_hu_h^l)\nabla_hu_h^l-\nabla_he_h^l,\label{nonlinear-term}
\end{align}
then, it follows that
\begin{align}
&\sum_{k,l}^{n,k}\theta_{k-l}^{(k)}\langle f(\nabla_h U^l)-f(\nabla_h u^l),\nabla_h e^k\rangle\nonumber\\
=&\sum_{k,l}^{n,k}\theta_{k-l}^{(k)}\langle|\nabla_hU^l|^2\nabla_he^l, \nabla_he^k\rangle\nonumber
+\sum_{k,l}^{n,k}\theta_{k-l}^{(k)}\langle(\nabla_hU^l\cdot\nabla_he^l)\nabla_hu^l, \nabla_h e^k\rangle\\
&+\sum_{k,l}^{n,k}\theta_{k-l}^{(k)}\langle(\nabla_hu^l\cdot\nabla_he^l)\nabla_hu^l, \nabla_h e^k\rangle-\sum_{k,l}^{n,k}\theta_{k-l}^{(k)}\langle\nabla_he^l, \nabla_h e^k\rangle.\label{nonlinear-inequality}
\end{align}
Next, we estimate each term on the right hand in the above equality.
For the first term on the right hand in (\ref{nonlinear-inequality}), Noticing $\||\nabla_hU^l|^2\|_{l^3}\le C_1^4|\Omega|^{\frac{2}{3}}:=C_2,$ and with the help of Lemma \ref{nonlinear-embedding inequality}, we obtain the following inequality for any $\varepsilon_1>0$
\begin{align}\label{inequality1}
\sum_{k,l}^{n,k}\theta_{k-l}^{(k)}\langle|\nabla_hU^l|^2\nabla_he^l, \nabla_h e^k\rangle
\le& \varepsilon_1\sum_{k=1}^{n}\tau_k(\|\Delta_he^k\|^2+\|\nabla_he^k\|^2)+\frac{C_\Omega C_2^2\mathfrak{m}_3}{4\mathfrak{m}_1\varepsilon_1}
\sum_{k=1}^{n}\tau_k\|\nabla_he^k\|^2.
\end{align}
Making direct use of Lemma \ref{lem: DOC property2}, one arrives the following estimate for  the second term on the right hand in (\ref{nonlinear-inequality})
\begin{align}\label{the third term inequality}
&\sum_{k,l}^{n,k}\theta_{k-l}^{(k)}\langle(\nabla_hU^l\cdot\nabla_he^l)\nabla_hu^l, \nabla_h e^k\rangle\nonumber\\
\le&\varepsilon_1\sum_{k=1}^{n}\tau_k\|(\nabla_hU^k\cdot\nabla_he^k)\nabla_h u^k\|^2
+\frac{\mathfrak{m}_3}{4\mathfrak{m}_1\varepsilon_1}\sum_{k=1}^{n}\tau_k\|\nabla_he^k\|^2.
\end{align}
By virtue of Cauchy-Schwarz inequality, Lemma \ref{bound-numerical-solution} and the inequality (\ref{l_4embedding inequality}), the following estimate holds for the first term on the right hand side of the above inequality
\begin{align*}
\|(\nabla_hU^k\cdot\nabla_he^k)\nabla_h u^k\|^2
&\le h^2\sum_{\mathbf{x}_h\in \Omega_h}|\nabla_hU_h^k|^2|\nabla_he_h^k|^2|\nabla_h u_h^k|^2\\
&=C_1^2\|\nabla_hu^k\|_{l^4}^2\cdot\|\nabla_he^k\|_{l^4}^2\\
&\le C_0^2C_1^2K_1^2\|\nabla_he^k\|\Big(2\|\Delta_he^k\|^2+\frac{1}{L^2}\|\nabla_he^k\|^2\Big)^{\frac{1}{2}}\\
&\le C_3(\|\nabla_he^k\|^2+\|\Delta_he^k\|^2), 
\end{align*}
where $C_3=C_0^2C_1^2K_1^2\max\{\frac{1}{2}+\frac{1}{2L^2}, 1\}.$
Then, substituting the above inequality into (\ref{the third term inequality}), it yields the final estimate of the second term on the right hand in (\ref{nonlinear-inequality})
\begin{align}
&\sum_{k,l}^{n,k}\theta_{k-l}^{(k)}\langle(\nabla_hU^l\cdot\nabla_he^l)\nabla_hu^l, \nabla_h e^k\rangle\nonumber\\
\le&C_3\varepsilon_1\sum_{k=1}^{n}\tau_k(\|\nabla_he^k\|^2+\|\Delta_he^k\|^2)
+\frac{\mathfrak{m}_3}{4\mathfrak{m}_1\varepsilon_1}\sum_{k=1}^{n}\tau_k\|\nabla_he^k\|^2. \label{nonlinear-embedding1}
\end{align}
For the third term on the right hand in (\ref{nonlinear-inequality}), by virtue of Lemma \ref{lem: DOC property2}, it is easily obtained that
\begin{align}
&\sum_{k,l}^{n,k}\theta_{k-l}^{(k)}\langle(\nabla_hu^l\cdot\nabla_he^l)\nabla_hu^l, \nabla_h e^k\rangle\nonumber\\
\le&\varepsilon_1\sum_{k=1}^n\tau_k\Big\|(\nabla_hu^k\cdot\nabla_he^k)\nabla_hu^k\Big\|^2
+\frac{\mathfrak{m}_3}{4\mathfrak{m}_1\varepsilon_1}\sum_{k=1}^{n}\tau_k\|\nabla_he^k\|^2.\label{the third inequality}
\end{align}
Making use of Cauchy-Schwarz inequality, it follows from (\ref{l6boundedness}) and (\ref{l_6embedding inequality}) that
\begin{align}
\Big\|(\nabla_hu^k\cdot\nabla_he^k)\nabla_hu^k\Big\|^2\le&h^2\sum_{\mathbf{x}_h\in \Omega_h}|\nabla_hu_h^k|^4|\nabla_he_h^k|^2\nonumber\\
\le&K^4\|\nabla_he^k\|_{l^6}^2\nonumber\\
\le&K^4K_2\|\nabla_he^k\|^{\frac{2}{3}}\Big(16\|\Delta_he^k\|^2+\frac{1}{L^2}\|\nabla_he^k\|^2\Big)^{\frac{2}{3}}\nonumber\\
\le&K^4K_2\Big[\frac{1}{3}\|\nabla_he^k\|^{2}+\frac{2}{3}\Big(16\|\Delta_he^k\|^2+\frac{1}{L^2}\|\nabla_he^k\|^2\Big)\Big]\nonumber\\
\le&C_4\sum_{k=1}^n\tau_k(\|\Delta_he^k\|^2+\|\nabla_he^k\|^2),\label{nonlinear-term3}
\end{align}
where $C_4=\max\{\frac{32}{3},\frac{1}{3}+\frac{2}{3L^2}\}.$
By inserting (\ref{nonlinear-term3}) into (\ref{the third inequality}), one gets the final estimate of the third term on the right hand in (\ref{nonlinear-inequality})
\begin{align}
&\sum_{k,l}^{n,k}\theta_{k-l}^{(k)}\langle(\nabla_hu^l\cdot\nabla_he^l)\nabla_hu^l, \nabla_h e^k\rangle\nonumber\\
\le&C_4\varepsilon_1\sum_{k=1}^n\tau_k(\|\Delta_he^k\|^2+\|\nabla_he^k\|^2)
+\frac{\mathfrak{m}_3}{4\mathfrak{m}_1\varepsilon_1}\sum_{k=1}^{n}\tau_k\|\nabla_he^k\|^2. \label{nonlinear-embedding11}
\end{align}
Applying Lemma \ref{lem:inequality of the kernel}, it yields the result for the fourth term on the right hand in (\ref{nonlinear-inequality})
\begin{align}
\absB{ -\sum_{k,l}^{n,k}\theta_{k-l}^{(k)}\langle\nabla_he^l, \nabla_h e^k\rangle}\le \frac{\mathfrak{m}_2}{2}\sum_{k=1}^n\tau_k\|\nabla_he^k\|^2.\label{first-estimate}
\end{align}
Substituting (\ref{inequality1}), (\ref{nonlinear-embedding1}), (\ref{nonlinear-embedding11})  and (\ref{first-estimate}) into (\ref{nonlinear-inequality}),
 one has an estimate of the nonlinear term
\begin{align}\label{nonlinear-term-estimate}
&\sum_{k,l}^{n,k}\theta_{k-l}^{(k)}\langle f(\nabla_h U^l)-f(\nabla_h u^l),\nabla_h e^k\rangle\nonumber\\
\le&~ C_5\varepsilon_1\sum_{k=1}^{n}\tau_k\|\Delta_he^k\|^2
+\Big(\frac{C_6}{\varepsilon_1}+C_5\varepsilon_1+\frac{\mathfrak{m}_3}{2}\Big)\sum_{k=1}^{n}\tau_k\|\nabla_he^k\|^2,
\end{align}
where $C_5=1+C_3+C_4,$ $C_{6}=\frac{\mathfrak{m}_3}{4\mathfrak{m}_1}(C_\Omega C_2^2+2).$

Summing up (\ref{BDF2-term-inner}), (\ref{equality2}) and (\ref{nonlinear-term-estimate}), then it follows from (\ref{inner error equality}) that
\begin{align}
&\frac{1}{2}\|e^n\|^2+\frac{\delta\mathfrak{m}_1}{2\mathfrak{m}_2}\sum_{k=1}^n\tau_k\|\Delta_he^k\|^2\nonumber\\
\le&\frac{1}{2}\|e^0\|^2+C_5\varepsilon_1\sum_{k=1}^{n}\tau_k\|\Delta_he^k\|^2
+\Big(\frac{C_6}{\varepsilon_1}+C_5\varepsilon_1+\frac{\mathfrak{m}_3}{2}\Big)\sum_{k=1}^{n}\tau_k\|\nabla_he^k\|^2
\nonumber\\
&+\sum_{k,l}^{n,k}\theta_{k-l}^{(k)}\langle\xi^l+\eta^l, e^k\rangle\nonumber\\
\le&\frac{1}{2}\|e^0\|^2+C_5\varepsilon_1\sum_{k=1}^{n}\tau_k\|\Delta_he^k\|^2
+\Big(\frac{C_6}{\varepsilon_1}+C_5\varepsilon_1+\frac{\mathfrak{m}_3}{2}\Big)\sum_{k=1}^{n}\tau_k\|\Delta_he^k\|\|e^k\|
\nonumber\\
&+\sum_{k,l}^{n,k}\theta_{k-l}^{(k)}\langle\xi^l+\eta^l, e^k\rangle\nonumber\\
\le&\frac{1}{2}\|e^0\|^2+\Big(C_5\varepsilon_1+\Big(\frac{C_6}{\varepsilon_1}+C_5\varepsilon_1+\frac{\mathfrak{m}_3}{2}\Big)\varepsilon_2\Big)
\sum_{k=1}^{n}\tau_k\|\Delta_he^k\|^2+\sum_{k=1}^n\|e^k\|(\|P^k\|+\|Q^k\|)\nonumber\\
&+\frac{1}{4\varepsilon_2}\Big(\frac{C_6}{\varepsilon_1}
+C_5\varepsilon_1+\frac{\mathfrak{m}_3}{2}\Big)\sum_{k=1}^{n}\tau_k\|e^k\|^2 ~~{\rm {for~any}~ \varepsilon_2>0}, \label{inequality9}
\end{align}
where we have used embedding inequality (\ref{embedding inequality}) and $P_h^k=\sum\limits_{l=1}^k\theta_{k-l}^{(k)}\xi_h^l$ and $Q_h^k=\sum\limits_{l=1}^k\theta_{k-l}^{(k)}\eta_h^l.$
Taking $\varepsilon_1=\frac{\delta\mathfrak{\mathfrak{m}_1}}{4C_5\mathfrak{m}_2}$
and $\varepsilon_2=\delta^2\mathfrak{m}_1^2/\Big(16\mathfrak{m}_2^2C_5C_6+\delta^2\mathfrak{m}_1^2+2\delta\mathfrak{m}_1\mathfrak{m}_2\mathfrak{m}_3\Big),$ it follows from (\ref{inequality9})
\begin{align*}
\|e^n\|^2\le&\|e^0\|^2+C_{7}\sum_{k=1}^{n}\tau_k\|e^k\|^2
+2\sum_{k=1}^n\|e^k\|(\|P^k\|+\|Q^k\|),
\end{align*}
where $C_{7}=\frac{1}{2\varepsilon_2}\Big(\frac{C_6}{\varepsilon_1}
+C_5\varepsilon_1+\frac{\mathfrak{m}_3}{2}\Big).$
Choosing a proper integer $n_0\,(0\le n_0\le n)$ such that $\|e^{n_0}\|=\max\limits_{1\le k\le n}\|e^k\|$ and then taking $n=n_0$ in the above inequality,
it yields
\begin{align*}
\|e^{n_0}\|^2\le\|e^0\|\cdot\|e^{n_0}\|+ C_{7}\sum_{k=1}^{n_0}\tau_k\|e^k\|\cdot\|e^{n_0}\|
+2\sum_{k=1}^{n_0}(\|P^k\|+\|Q^k\|)\cdot\|e^{n_0}\|,
\end{align*}
which leads to
 \begin{align*}
\|e^n\|&\le\|e^{n_0}\|\le\|e^0\|+C_{7}\sum_{k=1}^{n}\tau_k\|e^k\|
+2\sum_{k=1}^{n}(\|P^k\|+\|Q^k\|).
\end{align*}
When $\tau_n\le \frac{1}{2C_{7}}:=\tau_0,$ Lemma \ref{lem: Gronwall inequality} implies that
\begin{align}
\|e^n\|\le 2\exp(C_{7}t_{n-1})\Big(\|e^0\|+2\sum_{k=1}^{n}(\|P^k\|+\|Q^k\|)\Big).\label{error-estimate1}
\end{align}
By virtue of Lemma \ref{lem: DOC property} and Lemma \ref{lem:BDF2-Consistency-Error}, the desired estimate is obtained from (\ref{error-estimate1}). The proof ends.
\end{proof}

\section{Numerical experiments}
\setcounter{equation}{0}
In this section, we provide two numerical examples to verify the convergence rate in time and the energy dissipation. Only simple iteration is used to solve the nonlinear algebra equations at each time level with the tolerance as $10^{-12}$ and the solution at previous level is chosen as the initial guess. We test the convergence rate on the graded meshes for the first example.
\begin{example}
Consider the MBE model $u_t+\delta\Delta^2u+f(\nabla_hu)=g(\mathbf{x},t),\mathbf{x}\in\Omega=(0,2\pi)^2$, $0< t\le 1$ with $\delta=0.1.$
We take the function $g(\mathbf{x},t)$ such that it has an exact solution $u(\mathbf{x},t)=\cos t\sin x\sin y.$
\end{example}

The example is to demonstrate the time accuracy of the variable-step BDF2 scheme (\ref{eq: fully BDF2 implicit scheme}) on the graded meshes.
Let $r$ be a positive integer, $ t_k=T\left(\frac{k}{N}\right)^r$ and $\tau_k=t_k-t_{k-1}.  $ Denote the discrete $L^2$ norm error $e(N):=\|U^N-u^N\|$ and the order of convergence in time direction is
defined by $\mathrm{Order}={\rm log_2}(e(N)/e(2N)).$
The number of the spatial grid points are fixed as $M=3000.$ We list the numerical results,
including the $L^2$ norm error $e(N),$ the order of convergence in time direction. From Table \ref{PFC-BDF2-Time-Error}, it is easily verified that the variable-step
BDF2 scheme (\ref{eq: fully BDF2 implicit scheme}) achieves second-order accuracy in time as proved in Theorem \ref{th3.2}.

\begin{table}[htb!]
\begin{center}
\caption{Errors and convergence rate of the variable-step BDF2 scheme \eqref{eq: fully BDF2 implicit scheme}}\label{PFC-BDF2-Time-Error} \vspace*{0.3pt}
\def\temptablewidth{0.5\textwidth}
{\rule{\temptablewidth}{0.5pt}}
\begin{tabular*}{\temptablewidth}{@{\extracolsep{\fill}}ccccccc}
  $N$     &$e(N)$     &Order  \\
  \midrule
  40    &1.03e-04	 &--	 \\
  80    &2.82e-05	 &1.87	\\
  160	&6.74e-06	 &2.06  \\
  320	&1.60e-06    &2.07
\end{tabular*}
{\rule{\temptablewidth}{0.5pt}}
\end{center}
\end{table}	
Next, adaptive time-stepping strategy, which is designed to capture the multi-scale behavior of the gradient flow, is utilized to compute the MBE model (\ref{cont: MBE model}) in the implementation of the variable-step BDF2 scheme. We adopt the commonly used time adaptive strategy of [\cite{Gomez-Hughes} Algorithm 1 ] to get the variation of the time steps. In detail, the time-step $\tau_{ada}$ is updated adaptively using the current step information $\tau_{cur}$ by the formula
$\tau_{ada}(e,\tau_{cur})=\min\{\sqrt{tol/e}\rho\tau_{cur}\},$
where $\rho$ is a default safety coefficient, $tol$ is a reference tolerance and
$e$ is the relative error at each time level. Moreover, $\tau_{\max}$ and $\tau_{\min}$ are  predetermined maximum and minimum time steps respectively.
In the following simulation, we choose the safety coefficient as $\rho=0.9$, the reference tolerance $tol=10^{-3}$, the maximum time step $\tau_{\max}=0.1$ and the minimum time step $\tau_{\min}=10^{-4}$.
\begin{algorithm}
\caption{Adaptive time-stepping strategy}
\label{Adaptive-Time-Strategy}
\begin{algorithmic}[1]
\Require{Set $\tau_1:=\tau_{\min}.$ Given $u^{n}$ and time step $\tau_{n}$}
\State Compute $u^{n+1}$ by using BDF2 scheme with time step $\tau_{n}$.
\State Calculate $e_{n+1}=\|u^{n+1}-u^{n}\|/\|u^{n+1}\|$.
\If {$e_{n}<tol$ or $\tau_n\le \tau_{\min}$}
 \If {$e_{n}<tol$}
\State Update time-step size $\tau_{n+1}\leftarrow\min\{\max\{\tau_{\min},\tau_{ada}\},2.414\tau_n,\tau_{\max}\}$.
 \Else
 \State Update time-step size $\tau_{n+1}\leftarrow \tau_{\min}$.
 \EndIf
\Else
\State Recalculate with time-step size $\tau_{n}\leftarrow\max\{\tau_{\min},\tau_{ada}\}$.
\State Goto 1
\EndIf
\end{algorithmic}
\end{algorithm}

In addition, we use the roughness measure function $R(t)$ as defined in \cite{2010+Wang+DCD}:
$$R(t)=\sqrt{\frac{1}{|\Omega|}\int_\Omega\big(u(x,y,t)-\overline{u}(x,y,t)\big)^2~ \rm d\mathbf{x}},$$
where $\overline{u}(x,y,t)=\frac{1}{|\Omega|}\int_\Omega u(x,y,t)\rm d\mathbf{x}.$ It will be tested also in the next example.

\begin{example}
Take $\Omega=(0,2\pi)^2.$  Consider the problem (\ref{cont: MBE model}) with the initial condition as follows
\begin{align*}
u(x,y,0)=0.1(\sin3x\sin2y+\sin5x\sin5y).
\end{align*}
\end{example}
\begin{figure} [htp]
\begin{minipage}{8cm}
\epsfig{file=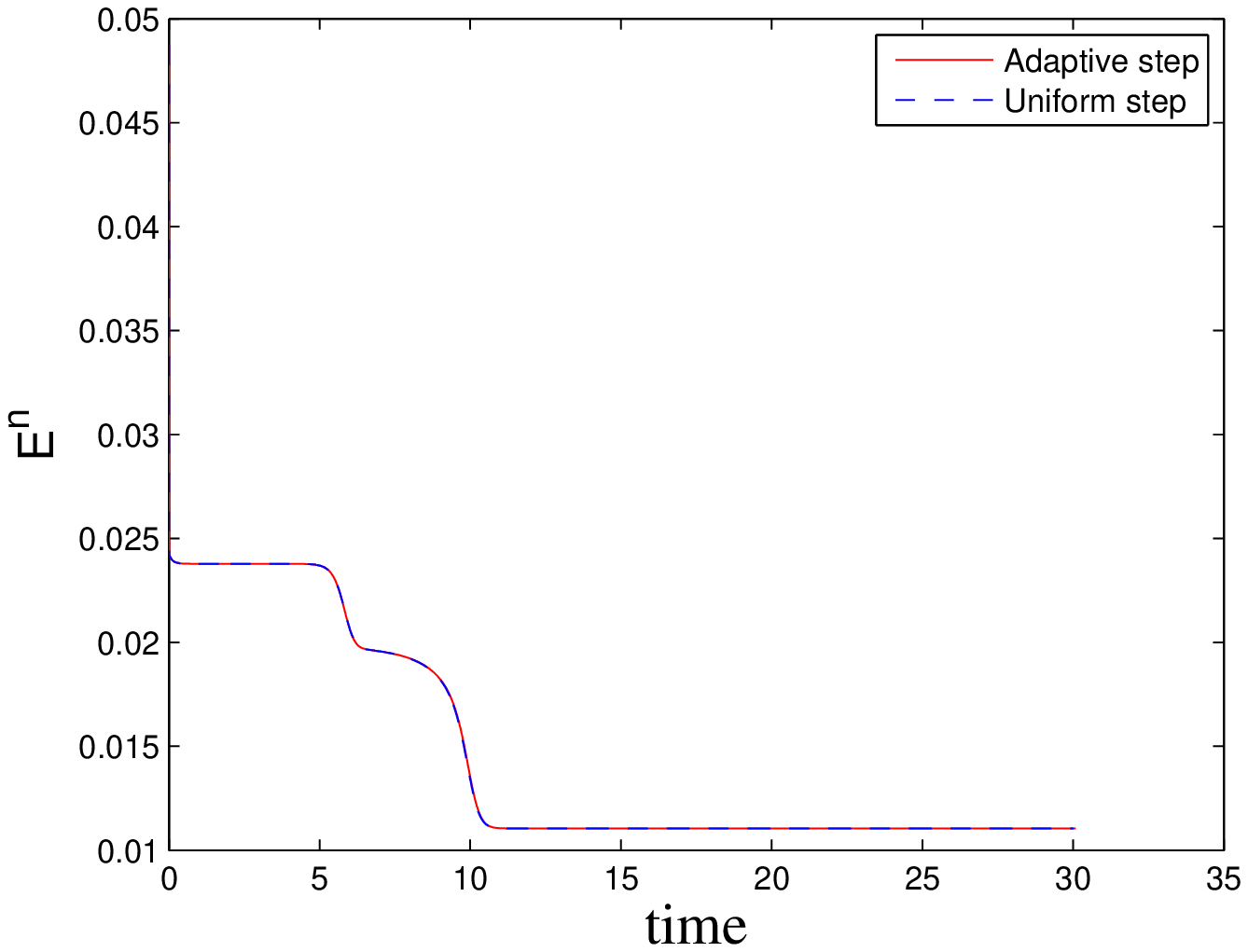, height=5cm}
\end{minipage}
\hfill
\begin{minipage}{8cm}
\epsfig{file=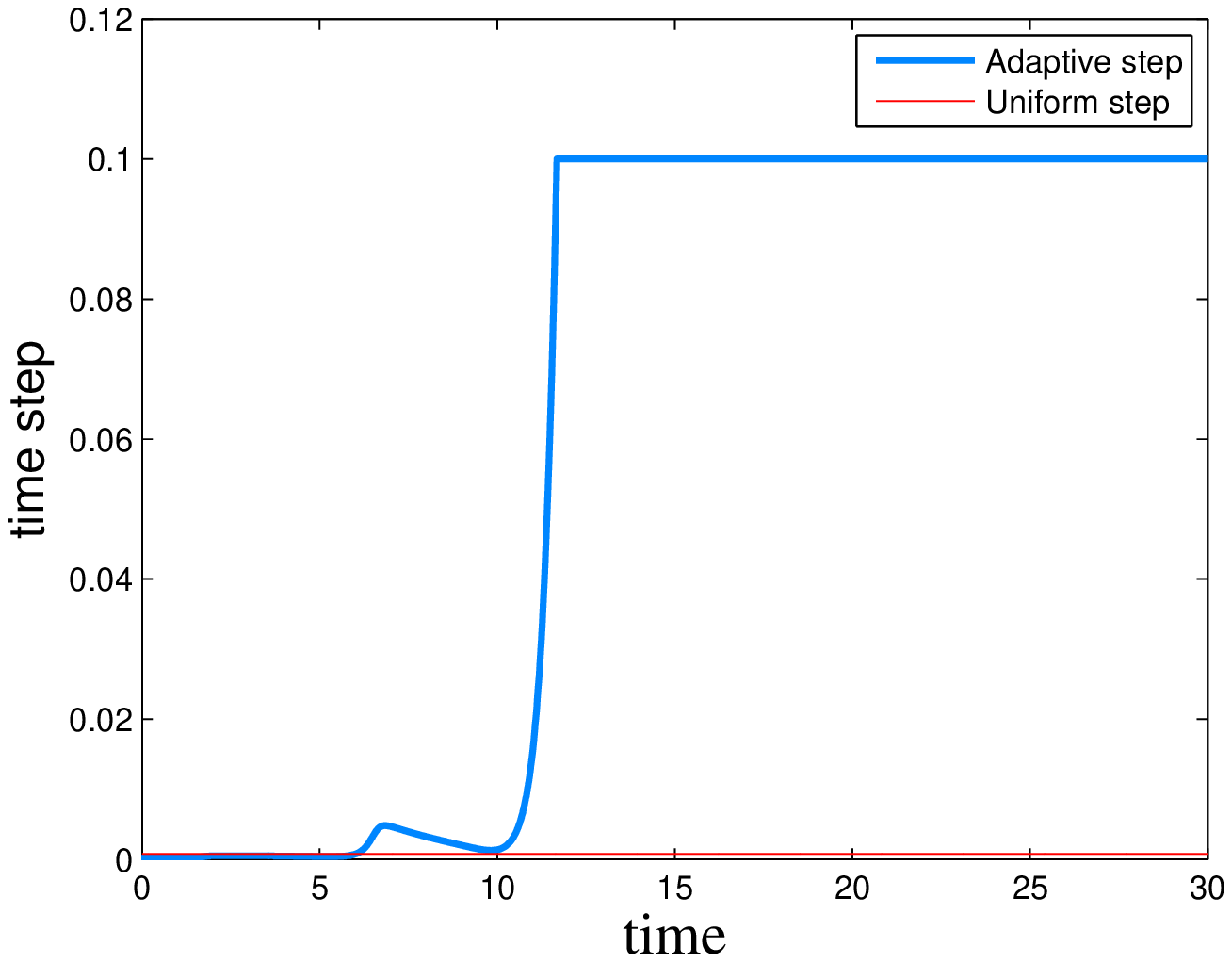, height=5cm}
\end{minipage}
\caption{Comparisons of the energy (left) and the time steps (right) of the BDF2 scheme (\ref{eq: fully BDF2 implicit scheme}) using the fixed time steps and the adaptive time
strategies.}\label{figure-energy1}
\end{figure}
We take the parameter $\delta=0.1$ and a $128\times 128$ uniform mesh
to discrete the spatial domain $\Omega=(0,2\pi)^2$.
In order to make the comparisons, we simulate the MBE model until $T=30$ on the uniform time meshes with the fixed time step $\tau=10^{-3}$ and the adaptive time meshes
(described in Algorithm 1), respectively. In Figure \ref{figure-energy1}, the time evolutions of discrete energies (left) and the corresponding time-step sizes (right) are depicted. From Figure \ref{figure-energy1}, we observe that the discrete energy curve on the adaptive time steps is in accordance with that generated by using a small constant step size. In addition, Figure \ref{figure-energy1} also demonstrates that the adaptive BDF2 scheme (\ref{eq: fully BDF2 implicit scheme}) is more efficient computationally that small time steps are chosen when the energy decays fast while large time steps are automatically selected when the energy dissipates slowly.


\begin{figure} [htp]
\begin{minipage}{8cm}
\epsfig{file=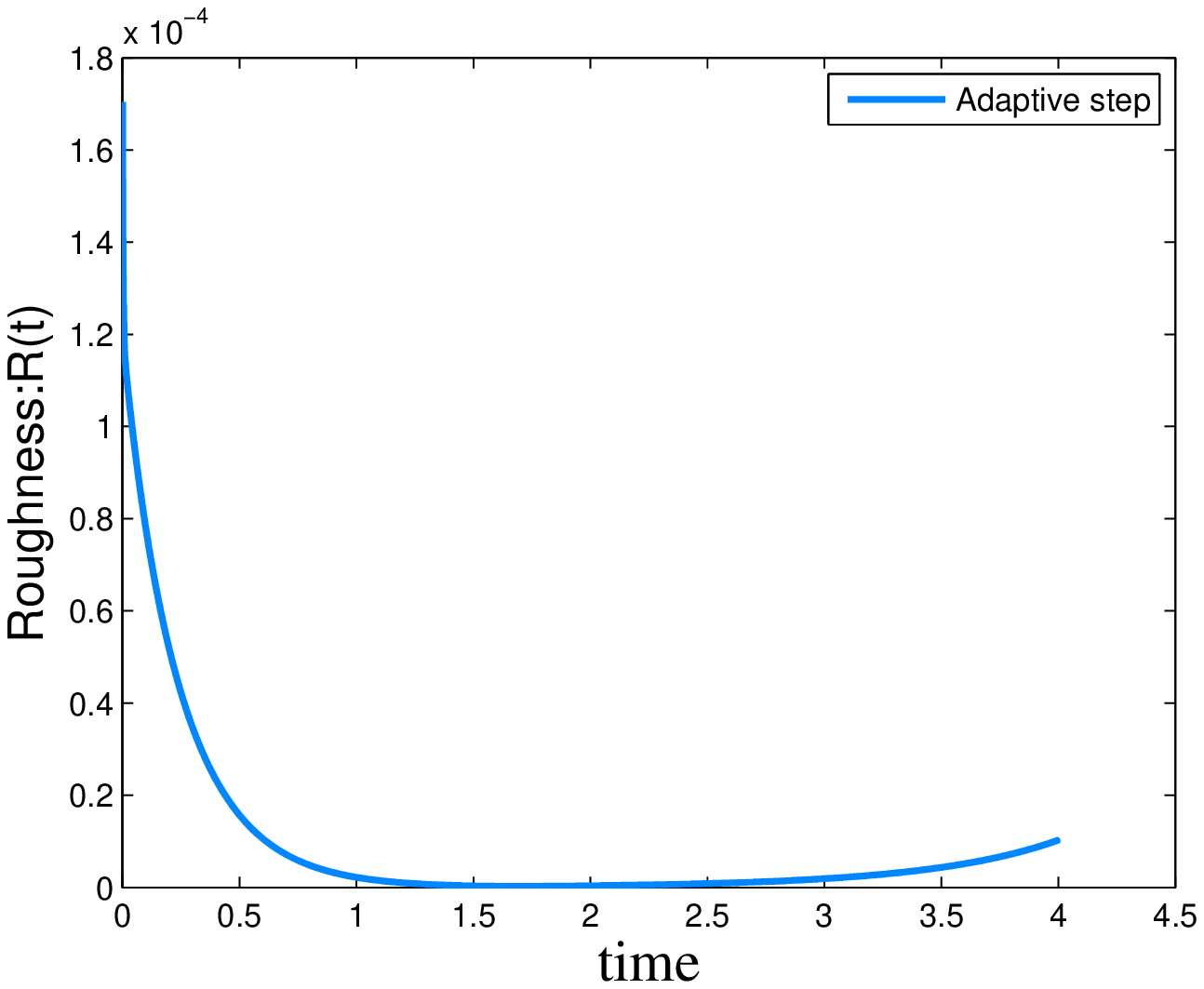, height=5cm}
\end{minipage}
\hfill
\begin{minipage}{8cm}
\epsfig{file=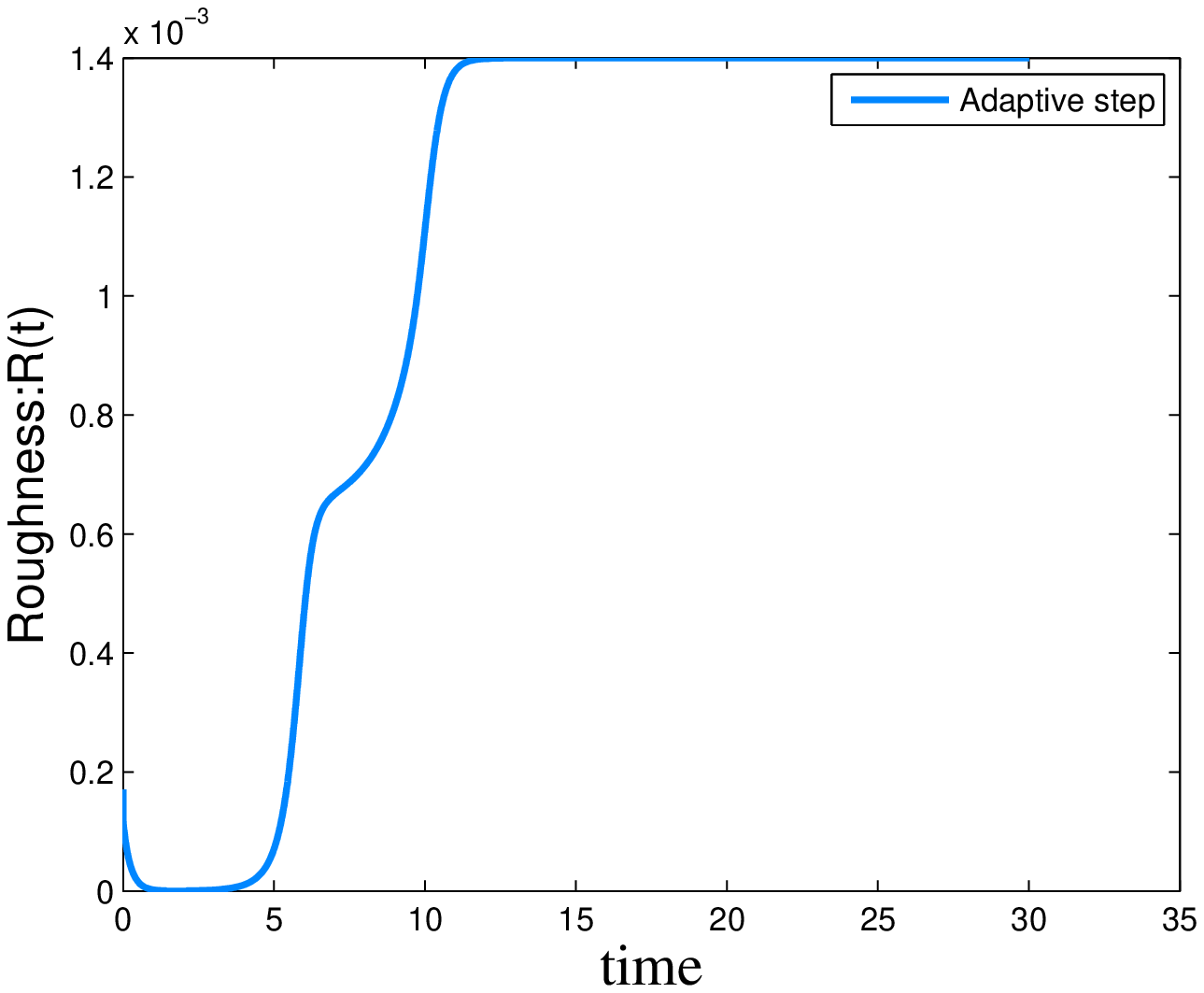, height=5cm}
\end{minipage}
\caption{Roughness evolution for MBE as $t=4$ (left) and $t=30$ (right) .}\label{Roughness}
\end{figure}

In Figure \ref{Roughness}, we plot the time evolution of the roughness measure function $R(t).$ The time evolution curve of $R(t)$ displays that
the deviation decreases for a short period and then keeps increasing until the steady-state.
In Figure \ref{snapshot-figure1}, the snapshots of the numerical solutions $u$ are shown until the steady-state by the adaptive BDF2 scheme (\ref{eq: fully BDF2 implicit scheme}).
\begin{figure} [htp]
\begin{minipage}{4.6cm}
\epsfig{file=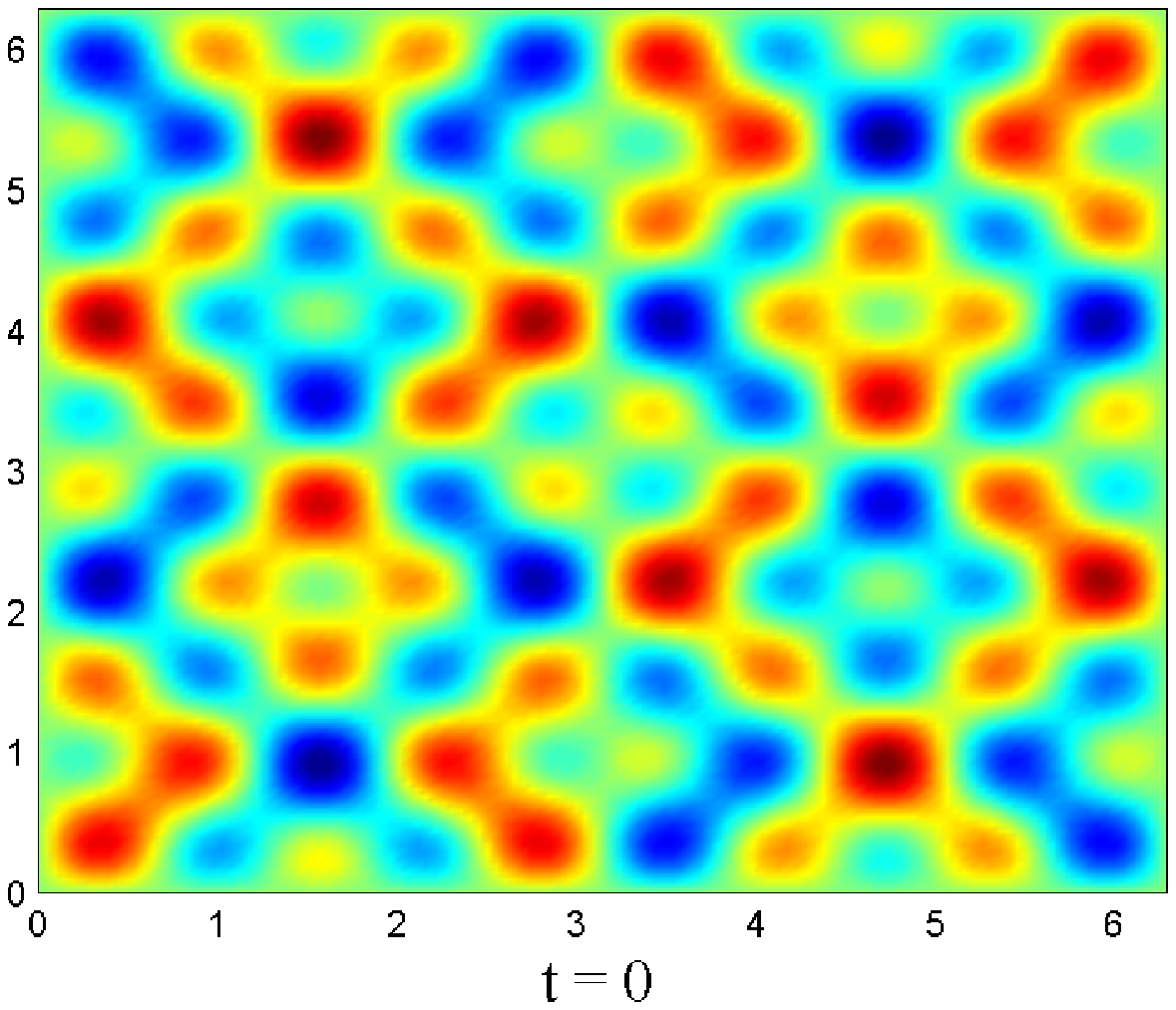, height=3.7cm}
\end{minipage}
\hfill
\begin{minipage}{4.6cm}
\epsfig{file=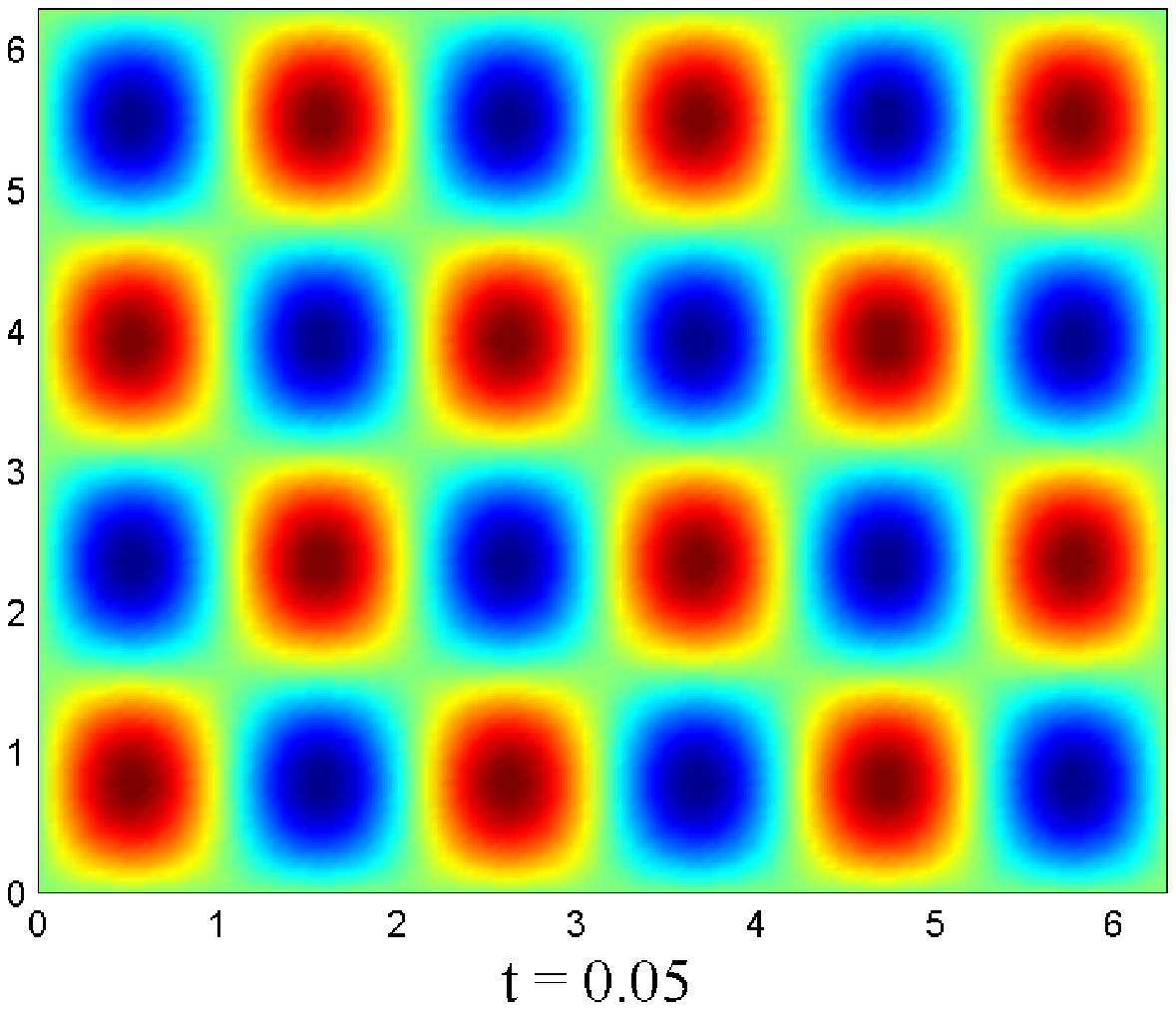, height=3.7cm}
\end{minipage}
\hfill
\begin{minipage}{4.6cm}
\epsfig{file=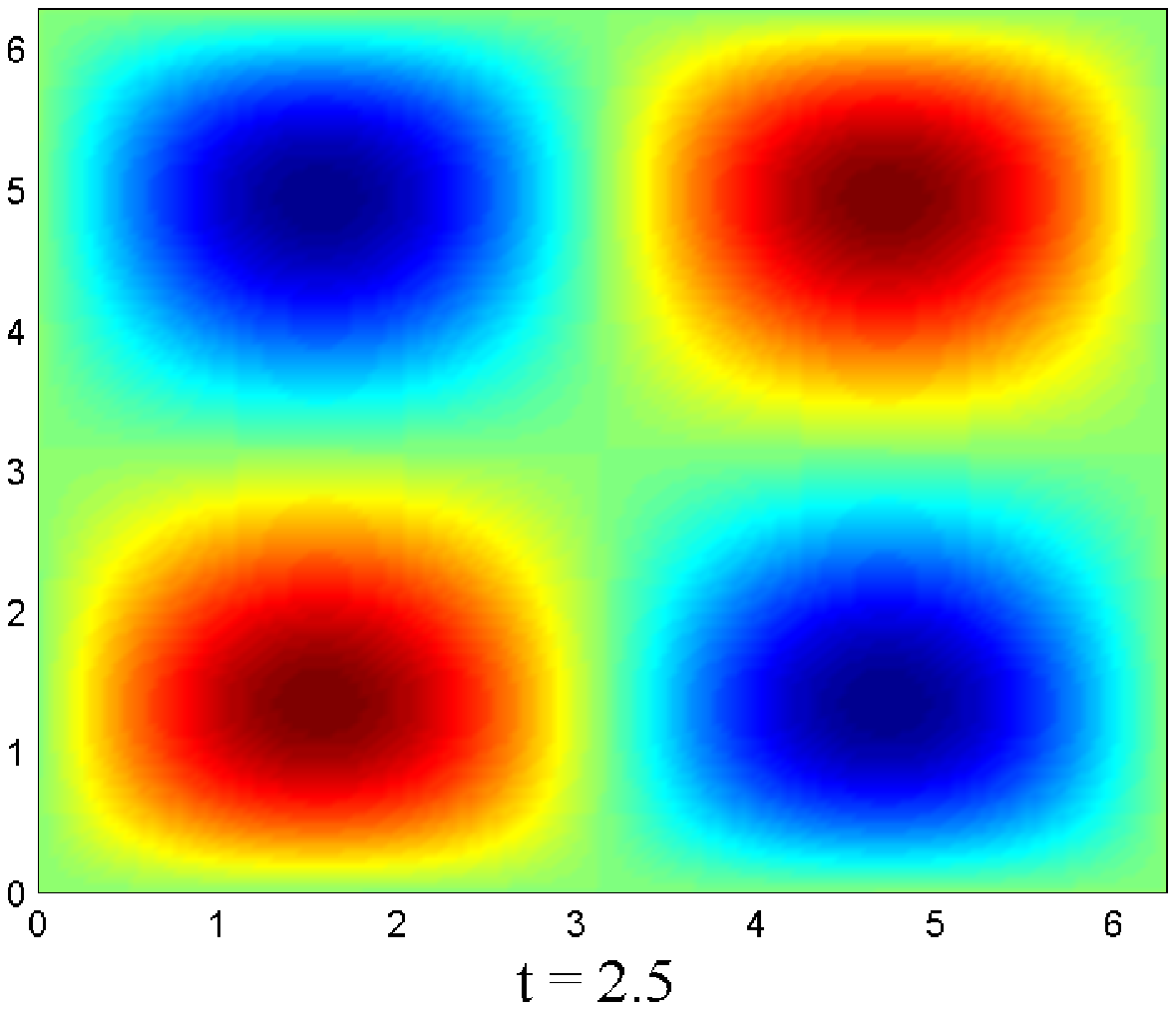, height=3.7cm}
\end{minipage}
\\
\begin{minipage}{4.6cm}
\epsfig{file=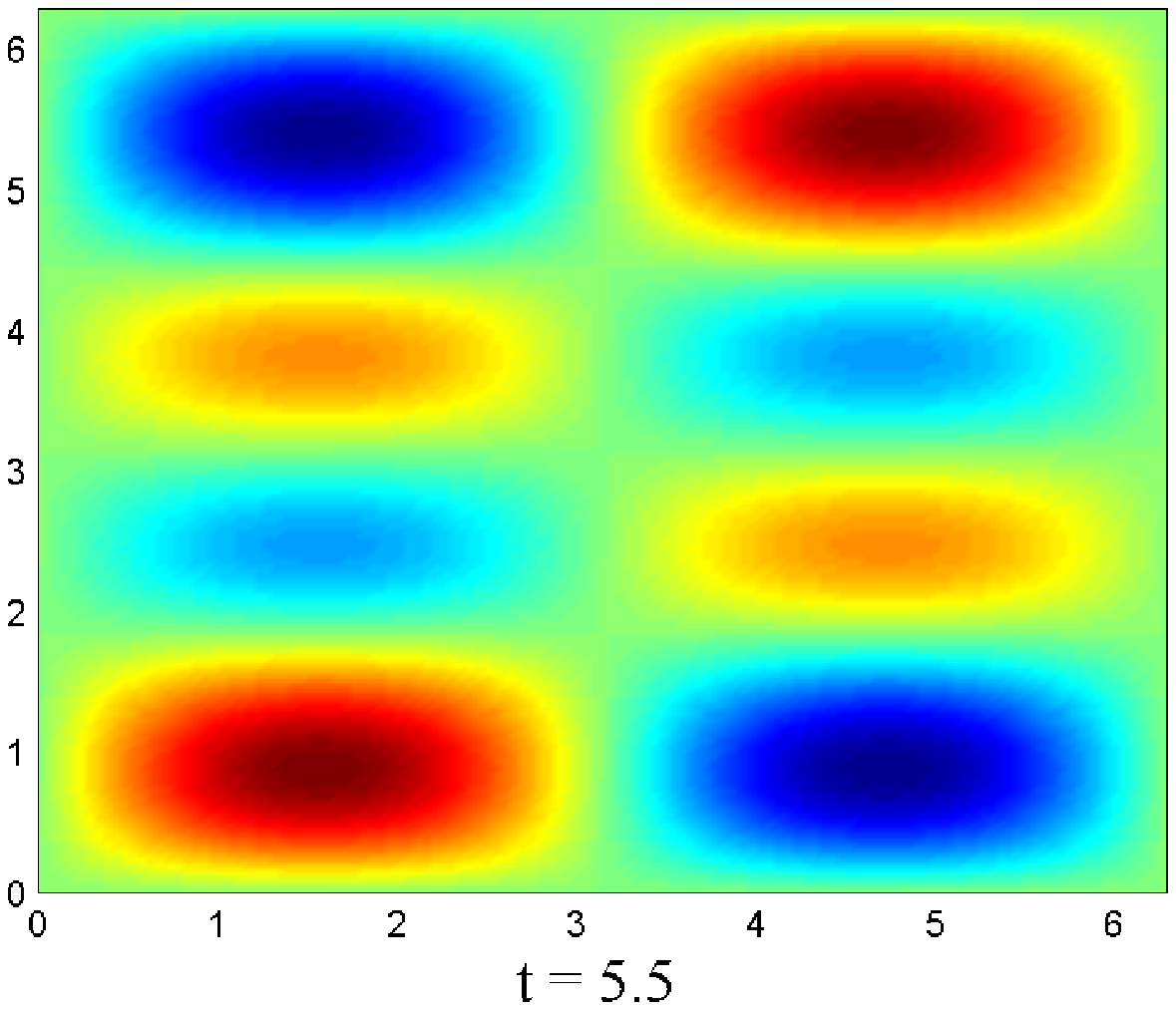, height=3.7cm}
\end{minipage}
\hfill
\begin{minipage}{4.6cm}
\epsfig{file=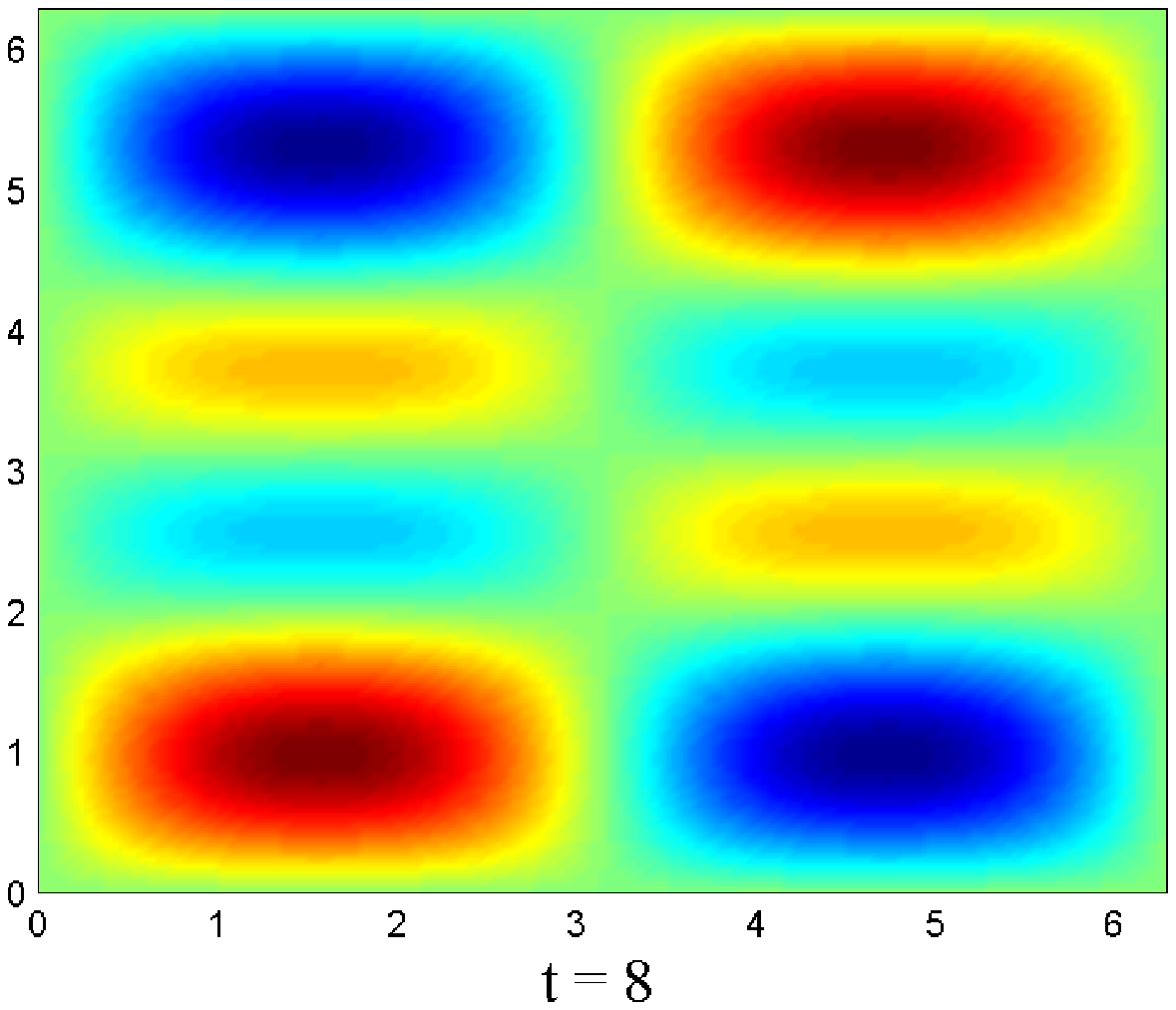, height=3.7cm}
\end{minipage}
\hfill
\begin{minipage}{4.6cm}
\epsfig{file=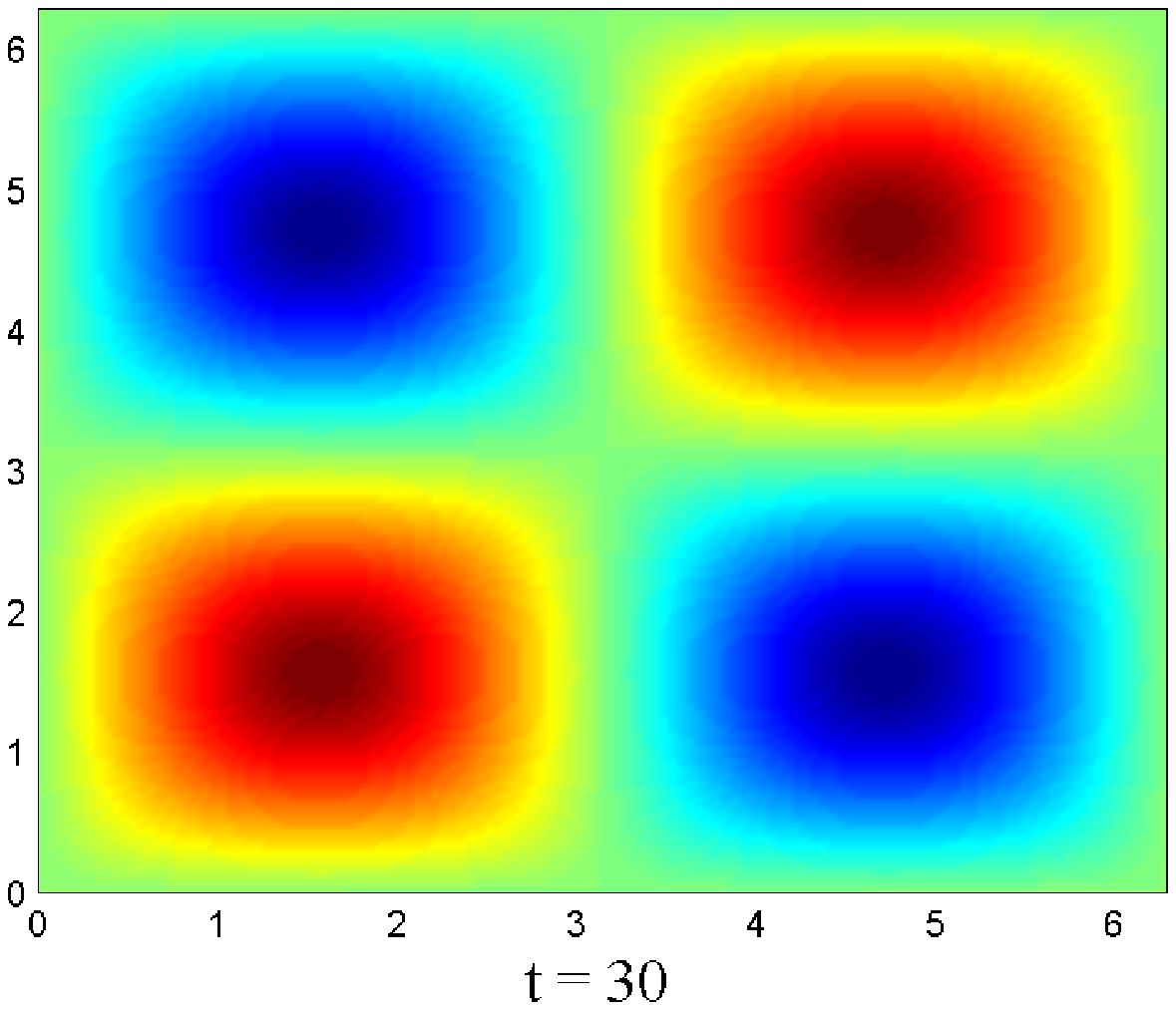, height=3.7cm}
\end{minipage}
\caption{Evolution of the solutions by the BDF2 scheme (\ref{eq: fully BDF2 implicit scheme})
using the adaptive time strategy at $t=0, 0.05, 2.5, 5.5, 8, 30.$ }\label{snapshot-figure1}
\end{figure}

\section*{Acknowledgement}
\noindent We would like to acknowledge support by the National Natural Science Foundation of China (No. 11701229,11701081,11861060), the Jiangsu Provincial Key Laboratory of Networked Collective Intelligence (No. BM2017002), Key Project of Natural Science Foundation of China (No. 61833005) and ZhiShan Youth Scholar Program of SEU,
 China Postdoctoral Science Foundation (No. 2019M651634), High-level Scientific
Research foundation for the introduction of talent of Nanjing Institute of Technology (No. YKL201856).

\section*{Appendix}
\def\theequation{A.\arabic{equation}}
We give the detailed proofs of Lemma \ref{lem:inequality of the kernel} - Lemma \ref{gradient-bound} in the Appendix. Two fundamental properties of the time discretization coefficients are provided in order to show the proofs. We start with the  introduction of the following two matrices
$$\mathbf{B_2}=\left(\begin{matrix}
   b_0^{(1)} &  & & \\
   b_1^{(2)} &b_0^{(2)}&  &  \\
   &  \ddots&\ddots&\\
   & & b_1^{(n)}&b_0^{(n)}
  \end{matrix}\right)\otimes \mathbf{I}_2,~~
\mathbf{\Theta_2}=\left(\begin{matrix}
   \theta_0^{(1)} &  &  & \\
   \theta_1^{(2)} &\theta_0^{(2)}&  &  \\
   \vdots&\ddots&\ddots&\\
  \theta_{n-1}^{(n)} & \ldots& \theta_1^{(n)}&\theta_0^{(n)}
  \end{matrix}\right)\otimes \mathbf{I}_2,$$
where the elements $b_{n-k}^{(n)}$ and $\theta_{n-k}^{(n)}$ are defined by (\ref{eq: BDF2-kernels}) and (\ref{eq: DOC-Kernels}), respectively,
$\mathbf{I}_2$ is $2\times2$ identity matrix and $\otimes$ is tensor product.
By virtue of the discrete orthogonal identity (\ref{eq: orthogonal identity}), it yields
$\mathbf{\Theta_2}=\mathbf{B_2^{-1}}.$
Denote $\mathbf{B}=\mathbf{B_2}+\mathbf{B_2^T}$ and $\mathbf{\Theta}=\mathbf{\Theta}_2+\mathbf{\Theta}_2^T,$ then it follows from Lemma \ref{lem:Conv-Kernels-Positive} and Lemma \ref{lem: DOC property} that
the matrix $\mathbf{B}$ and $\mathbf{\Theta}$ are symmetric and positive definite. By substitution of $\mathbf{\Theta}$ with $\mathbf{B_2},$ we have
\begin{align}
\mathbf{\Theta}:=\mathbf{B}_2^{-1}+(\mathbf{B}_2^{-1})^T=(\mathbf{B}_2^{-1})^T\mathbf{B}\mathbf{B}_2^{-1}.\label{theta-decompose}
\end{align}
In addition, we define a diagonal matrix $\mathbf{\Lambda}_\tau={\rm diag}(\sqrt{\tau_1}, \sqrt{\tau_2}, \cdots, \sqrt{\tau_n})\otimes \mathbf{I}_2$ and denote
$$\widetilde{\mathbf{B}}_2:=\mathbf{\Lambda}_\tau \mathbf{B_2}\mathbf{\Lambda}_\tau=\widetilde{B}_2\otimes\mathbf{I}_2{~~\rm{and}~~
~~\mathbf{\widetilde{\Theta}_2}:=\widetilde{B}_2^{-1}\otimes\mathbf{I}_2,}
$$
where $\widetilde{B}_2=\left(\begin{matrix}
   \tilde{b}_0^{(1)} &  &  & \\
   \tilde{b}_1^{(2)} &\tilde{b}_0^{(2)}&  &  \\
   &\ddots&\ddots&\\
    & & \tilde{b}_1^{(n)}&\tilde{b}_0^{(n)}
  \end{matrix}\right)$
{ and $\widetilde{\Theta}_2:=\widetilde{B}^{-1}_2=\left(\begin{matrix}
   \tilde{\theta}_0^{(1)} &  &  & \\
   \tilde{\theta}_1^{(2)} &\tilde{\theta}_0^{(2)}&  &  \\
  \vdots &\vdots&\ddots&\\
   \tilde{\theta}_{n-1}^{(n)}&\tilde{\theta}_{n-2}^{(n)} & \cdots&\tilde{\theta}_0^{(n)}
  \end{matrix}\right).$}

It is easy to check that $\tilde{b}^{(1)}_0=2,$
$ \tilde{b}_0^{(k)}=\frac{1+2r_k}{1+r_k}~~{\rm and}~~ \tilde{b}_1^{(k)}=-\frac{r_k^{\frac{3}{2}}}{1+r_k},~~2\le k\le n.$
{It follows from Lemma \ref{lem: DOC property} that
\begin{align}
\tilde{\theta}_{k-j}^{(k)}=\frac{1}{\sqrt{\tau_k\tau_j}}\theta_{k-j}^{(k)}=\frac{1+r_j}{1+2r_j}\prod_{i=j+1}^k\frac{r_i^{\frac{3}{2}}}{1+2r_i},
~~1\le j\le k\le n.\label{step-scaled DOC}
\end{align}}
 Define the symmetric matrix $\widetilde{\mathbf{B}}:=\widetilde{\mathbf{B}}_2+\widetilde{\mathbf{B}}_2^T=\Lambda_\tau\mathbf{B}\Lambda_\tau.$
We introduce the vector norm $\il\cdot\il$ by $\il\bs{u}\il=\sqrt{\bs{u}^T\bs{u}}$ and the associated matrix norm
$\il\mathbf{A}\il:=\sqrt{\rho(\mathbf{A^T}\mathbf{A})}.$

{\bf The proof of Lemma \ref{lem:inequality of the kernel}}~~~~{Firstly, we estimate the lower bound of $\lambda_{\min}(\widetilde{\mathbf{B}})$ and the upper bound of $\lambda_{\max}(\widetilde{\mathbf{B}}_2^T\widetilde{\mathbf{B}}_2).$
Denote $\lambda$ be the eigenvalue of the matrix $\widetilde{\mathbf{B}}$. By use of the Gerschgorin's circle theorem, one arrives at
$|\lambda-2\tilde{b}_0^{(k)}|\le |\tilde{b}_1^{(k)}|+|\tilde{b}_1^{(k+1)}|,~~2\le k\le n-1.$
Then it yields
\begin{align}
\lambda_{\min}(\widetilde{\mathbf{B}})
\ge\min_{1\le k\le n}\Big\{\frac{2+4r_k-r_k^{\frac{3}{2}}}{1+r_k}-\frac{r_{k+1}^\frac{3}{2}}{1+r_{k+1}}\Big\}
&\ge\frac{2+2r_s(2-\sqrt{r_s})}{1+r_s}:=\mathfrak{m}_1.\label{lowerbound}
\end{align}

By applying the properties of the Kronecker tensor product $(A\otimes B)^T=A^T\otimes B^T$ and $(A\otimes B)(C\otimes D)=AC\otimes BD,$ we have
$
\widetilde{\mathbf{B}}_2^T\widetilde{\mathbf{B}}_2=(\widetilde{B}_2^T\widetilde{B}_2)\otimes\mathbf{I}_2.
$
Making use of the Gerschgorin's circle theorem, it leads to
\begin{align*}
\lambda_{\max}(\widetilde{B}_2^T\widetilde{B}_2)\le \max_{2\le k\le n}\{\mathfrak{R}(r_k,r_{k+1}), 3+\mathfrak{R}(r_1,r_{2})\}\le 3+\mathfrak{R}(r_s,r_{s}):=\mathfrak{m}_2,
\end{align*}
where $r_1=0$ and the function $\mathfrak{R}(u,v)$ is defined by
$$\mathfrak{R}(u,v)=\frac{(1+2u)(1+2u+u^{\frac{3}{2}})}{(1+u)^2}+\frac{v^{\frac{3}{2}}(1+2v+v^{\frac{3}{2}})}{(1+v)^2},~~0\le u,v\le r_s.$$
Then we obtain
\begin{align}
\lambda_{\max}(\widetilde{\mathbf{B}}_2^T\widetilde{\mathbf{B}}_2)\le\mathfrak{m}_2.\label{upperbound}
\end{align}

Denote $\bs{v}=((\bs{v}^1)^T, (\bs{v}^2)^T, \cdots, (\bs{v}^n)^T)^T.$ 
The inequality (\ref{lowerbound}) implies that the symmetric matrix $\widetilde{\mathbf{B}}$ is positive definite. There exists a non-singular upper
triangular matrix $\widetilde{\mathbf{U}}$ such that $\widetilde{\mathbf{B}}=\widetilde{\mathbf{U}}^T\widetilde{\mathbf{U}}.$
Then we obtain
\begin{align*}
\bs{v}^T \bs {\Theta}\bs{v}=\bs{v}^T(\bs{B}_2^{-1})^T\bs{B}\bs{B}_2^{-1}\bs{v}=\bs{v}^T(\bs{B}_2^{-1})^T\bs{\Lambda}_\tau^{-1}
\bs{\widetilde{B}}\bs{\Lambda}_\tau^{-1}\bs{B}_2^{-1}\bs{v}=\il\widetilde{\mathbf{U}}\bs{\Lambda}_\tau^{-1}\bs{B}_2^{-1}\bs{v}\il^2.
\end{align*}
Consequently, it yields
\begin{align*}
\il\bs{\Lambda}_\tau\bs v\il^2&=\il\bs{\Lambda}_\tau\bs{B}_2\bs{\Lambda}_\tau\widetilde{\mathbf{U}}^{-1}\widetilde{\mathbf{U}}
\bs{\Lambda}_\tau^{-1}\bs{B}_2^{-1}\bs{v}\il^2\\
&\le \il\bs{\widetilde{B}}_2\widetilde{\mathbf{U}}^{-1}\il^2
\il\widetilde{\mathbf{U}}\bs{\Lambda}_\tau^{-1}\bs{B}_2^{-1}\bs{v}\il^2\\
&\le \il\bs{\widetilde{B}}_2\il^2\il\widetilde{\mathbf{U}}^{-1}\il^2\bs{v}^T\bs{\Theta}\bs{v}
\\&=\lambda_{\max}(\bs{\widetilde{B}}_2^T\bs{\widetilde{B}}_2)\lambda_{\max}(\bs{\widetilde{B}}^{-1})\bs{v}^T\bs{\Theta}\bs{v}.
\end{align*}
It follows from (\ref{lowerbound}) and (\ref{upperbound}) that
\begin{align}
 \bs v^T \mathbf{\Theta}\bs v\ge\frac{\mathfrak{m}_1}{\mathfrak{m}_2}.\label{lower-bound}
\end{align}

Let $\widetilde{\Theta}=\widetilde{\Theta}_2+\widetilde{\Theta}_2^T.$ Noticing $0<\frac{x^{3/2}}{1+2x}<m_*:\frac{r_{s}^{3/2}}{1+2r_{s}}<1$ for any
$x\in[0,r_s),$ we have
\begin{align*}
\mathfrak{R}_{n,k}:=\sum_{j=1}^k\widetilde{\theta}_{k-j}^{(k)}+\sum_{j=k}^{n}\theta_{j-k}^{(j)}\le \sum_{j=1}^k m_*^{k-j}+\sum_{j=k}^n m_*^{j-k}<\frac{2}{1-m_*},~~1\le k\le n.
\end{align*}
 Using Gerschgorin's circle theorem, one arrives at $\lambda_{\max}(\widetilde{\Theta})\le \max_{1\le k\le n}\mathfrak{R}_{n,k}<\mathfrak{m}_3:=\frac{2}{1-m_*}.$ Thus it leads to $\bs{w}^T\bs{\widetilde{\Theta}}\bs{w}\le \mathfrak{m}_3\il \bs{w}\il^2 $
for any $\bs{w},$ where $\bs{\widetilde{\Theta}}:=\widetilde{\Theta}\otimes \bs{I}_2.$ By taking $\bs{w}:=\bs{\Lambda}_\tau\bs{v},$ it yields
\begin{align}
\bs v^T\mathbf{\Theta}\bs v\le \mathfrak{m}_3\il \mathbf{\Lambda}_\tau \bs v\il^2.\label{upper-bound}
\end{align}
Combining (\ref{lower-bound}) and(\ref{upper-bound}) and using the fact $\bs v^T \mathbf{\Theta}\bs v=2\sum\limits_{k,l}^{n,k}\theta_{k-l}^{(k)}(\bs v^k)^T\bs v^l,$ it yields the desired result.}

{\bf The proof of Lemma \ref{lem: DOC property2}}~~~~{Let $\bs{w}=((\bs{w}^1)^T, (\bs{w}^2)^T,\cdots, (\bs{w}^n)^T)^T.$
 A similar proof of Lemma A.3 in \cite{2022-Ji+JSC} yields
 \begin{align*}
 \sum_{k,j}^{n,k}\theta_{k-j}^{(k)}({\bs v}^j)^T{\bs w}^k\le \varepsilon \sum_{k,j}^{n,k}\theta_{k-j}^{(k)}({\bs v}^j)^T{\bs v}^k+
 \frac{1}{2\varepsilon}{\bs w}^T{\bs B}^{-1}{\bs w}~ {\rm{for~any~}} \varepsilon>0.
\end{align*}
It is easy to obtain $\bs{B}^{-1}=\bs{\Lambda}_\tau\bs{\widetilde{U}}^{-1}(\bs{\Lambda}_\tau\bs{\widetilde{U}}^{-1})^T.$ It follows that
\begin{align*}
\bs{w}^T\bs{B}^{-1}\bs{w}
&=\il(\bs{\widetilde{U}}^{-1})^T\bs{\Lambda}_\tau \bs{w}\il^2\le \il(\bs{\widetilde{U}}^{-1})^T\il^2\il\bs{\Lambda}_\tau \bs{w}\il^2\\
&=\lambda_{\max}(\bs{\widetilde{B}}^{-1})\bs{w}^T\bs{\Lambda}_\tau^2\bs{w}\le \mathfrak{m}_1^{-1}\sum_{k=1}^n\tau_k(\bs{w}^k)^2.
\end{align*}
Consequently, we have
\begin{align*}
 \sum_{k,j}^{n,k}\theta_{k-j}^{(k)}({\bs v}^j)^T{\bs w}^k\le \varepsilon \sum_{k,j}^{n,k}\theta_{k-j}^{(k)}({\bs v}^j)^T{\bs v}^k+
 \frac{1}{2\mathfrak{m}_1\varepsilon}\sum_{k=1}^n\tau_k(\bs{w}^k)^2~ {\rm{for~any~}} \varepsilon>0.
\end{align*}
By choosing $\varepsilon=2\epsilon/\mathfrak{m}_3$ in the above inequality and using (\ref{upper-bound}), one gets
\begin{align*}
\sum_{k,j}^{n,k}\theta_{k-j}^{(k)}({\bs v}^j)^T{\bs w}^k\le\epsilon\sum_{k=1}^{n}\tau_k(\bs{v}^k)^T\bs{v}^k+
\frac{\mathfrak{m}_{3}}{4\mathfrak{m}_1\epsilon}\sum_{k=1}^{n}\tau_k(\bs{w}^k)^T\bs{w}^k.
\end{align*}

}
{\bf The proof of Lemma \ref{nonlinear-embedding inequality}} For the fixed time index $n,$ by taking $\bs{v}^j:=u^j\bs{z}^j$ and $\varepsilon:=\varepsilon_1$ in Lemma \ref{lem: DOC property2}, it yields
\begin{align*}
\sum_{k,j}^{n,k}\theta_{k-j}^{(k)}\langle u^j\bs{z}^j,\bs{w}^k\rangle\le \varepsilon_1\sum_{k=1}^n\tau_k\|u^k\bs{z}^k\|^2
+\frac{\mathfrak{m}_3}{4\mathfrak{m}_1\varepsilon_1}\sum_{k=1}^{n}\tau_k\|\bs{w}^k\|^2.
\end{align*}
With the help of H\"{o}lder inequality, one has $\|u^k\bs{z}^k\|\le\|u^k\|_{l^3}\|\bs{z}^k\|_{l^6}.$
The embedding inequality $\|\bs z^k\|_{l^6}^2\le C_\Omega(\|\nabla_h\bs z^k\|^2+\|\bs z^k\|^2)$ yields
\begin{align*}
\sum_{k=1}^n\tau_k\|u^k\bs z^k\|^2&\le\sum_{k=1}^n\tau_k\|u^k\|^2_{l^3}\|\bs z^k\|_{l^6}^2\\&\le C_\Omega\sum_{k=1}^n\tau_k\|u^k\|_{l^3}^2(\|\nabla_h\bs z^k\|^2+\|\bs z^k\|^2)\\
&\le C_u^2C_\Omega\sum_{k=1}^n\tau_k(\|\nabla_h\bs z^k\|^2+\|\bs z^k\|^2).
\end{align*}
It follows from the above inequality that
\begin{align}\label{lem:inequality}
\sum_{k,j}^{n,k}\theta_{k-j}^{(k)}\langle u^j\bs z^j,\bs w^k\rangle\le \varepsilon_1C_u^2C_\Omega\sum_{k=1}^n\tau_k
(\|\nabla_h\bs z^k\|^2+\|\bs z^k\|^2)
+\frac{\mathfrak{m}_3}{4\mathfrak{m}_1\varepsilon_1}\sum_{k=1}^{n}\tau_k\|\bs{w}^k\|^2.
\end{align}
Taking $\varepsilon=\varepsilon_1C_u^2C_\Omega$, one gets the claimed inequality.

{\bf The proof of Lemma \ref{gradient-bound}}~~~{
For any integers $1\le s\le m\le M,$ we have
\begin{align*}
&|\Delta_xu_{mj}|^3-|\Delta_xu_{sj}|^3
=\sum_{i=s}^{m-1}\braB{|\Delta_xu_{i+1,j}|^3-|\Delta_xu_{ij}|^3}\\
=&\sum_{i=s}^{m-1}\braB{|\Delta_xu_{i+1,j}|-|\Delta_xu_{ij}|}\braB{|\Delta_xu_{i+1,j}|^2
+|\Delta_xu_{ij}|\cdot|\Delta_xu_{i+1,j}|+|\Delta_xu_{ij}|^2}\\
\le&2h\sum_{i=s}^{m-1}\absB{\frac{\Delta_xu_{i+1,j}-\Delta_xu_{ij}}{h}}\cdot
\braB{|\Delta_xu_{i+1,j}|^2+|\Delta_xu_{ij}|^2}\\
=&h\sum_{i=s}^{m-1}\absb{\delta_x^2u_{i+1,j}+\delta_x^2u_{ij}}\cdot
\braB{|\Delta_xu_{i+1,j}|^2+|\Delta_xu_{ij}|^2}\\
\le&4\braB{h\sum_{i=1}^{M}\absb{\delta_x^2u_{ij}}^2}^{\frac{1}{2}}\braB{h\sum_{i=1}^{M}\absb{\Delta_xu_{ij}}^4}^{\frac{1}{2}}.
\end{align*}
It is easy to verify the above inequality also holds for $m\le s$. Then, it follows that
$$|\Delta_xu_{mj}|^3\le4\braB{h\sum_{i=1}^{M}\absb{\delta_x^2u_{ij}}^2}^{\frac{1}{2}}
\braB{h\sum_{i=1}^{M}\absb{\Delta_xu_{ij}}^4}^{\frac{1}{2}}+|\Delta_xu_{sj}|^3,~~1\le m, s\le M. $$

Multiplying the above inequality by $h$ and summing up $s$ from 1 to $M,$ it yields
$$L|\Delta_xu_{mj}|^3\le4L\braB{h\sum_{i=1}^{M}\absb{\delta_x^2u_{ij}}^2}^{\frac{1}{2}}
\braB{h\sum_{i=1}^{M}\absb{\Delta_xu_{ij}}^4}^{\frac{1}{2}}+h\sum_{i=1}^M|\Delta_xu_{ij}|^3.$$
The above inequality holds for $m=1,2,\ldots, M,$ one can get
$$\max_{1\leq m\leq M}|\Delta_xu_{mj}|^3\le4\braB{h\sum_{i=1}^{M}\absb{\delta_x^2u_{ij}}^2}^{\frac{1}{2}}
\braB{h\sum_{i=1}^{M}\absb{\Delta_xu_{ij}}^4}^{\frac{1}{2}}+\frac{1}{L}h\sum_{i=1}^M|\Delta_xu_{ij}|^3.$$
Multiplying the above inequality by $h$, and summing up $j$ from 1 to $M,$ we have
\begin{align}
h\sum_{j=1}^M\max_{1\leq m\leq M}|\Delta_xu_{mj}|^3
&\le 4h\sum_{j=1}^M\braB{h\sum_{i=1}^{M}\absb{\delta_x^2u_{ij}}^2}^{\frac{1}{2}}
\braB{h\sum_{i=1}^{M}\absb{\Delta_xu_{ij}}^4}^{\frac{1}{2}}+\frac{1}{L}\|\Delta_xu\|_3^3\nonumber\\
&\le4\braB{h^2\sum_{i=1}^{M}\sum_{j=1}^{M}\absb{\delta_x^2u_{ij}}^2}^{\frac{1}{2}}
\braB{h^2\sum_{i=1}^{M}\sum_{j=1}^{M}\absb{\Delta_xu_{ij}}^4}^{\frac{1}{2}}+\frac{1}{L}\|\Delta_xu\|_3^3.\label{ineq:bound1}
\end{align}
By Cauchy-Schwarz inequality, we obtain
\begin{align*}
 \|\Delta_xu\|_3^3 \le \braB{h^2\sum_{i=1}^{M}\sum_{j=1}^{M}\absb{\Delta_xu_{ij}}^2}^\frac{1}{2}\cdot
 \braB{h^2\sum_{i=1}^{M}\sum_{j=1}^{M}\absb{\Delta_xu_{ij}}^4}^\frac{1}{2}
 =\|\Delta_x u\|\cdot\|\Delta_x u\|_4^2.
\end{align*}
Substituting the above inequality into (\ref{ineq:bound1}), we have
\begin{align}
h\sum_{j=1}^M\max_{1\leq m\leq M}|\Delta_xu_{mj}|^3\leq \|\Delta_x u\|_4^2\braB{4\|\delta_x^2u\|+\frac{1}{L}\|\Delta_x u\|}.\label{ineq:bound2}
\end{align}
Besides, it is also valid
\begin{align*}
|\Delta_xu_{im}|^3-|\Delta_xu_{is}|^3
=\sum_{j=s}^{m-1}\braB{|\Delta_xu_{i,j+1}|^3-|\Delta_xu_{ij}|^3}.
\end{align*}
Similar to the previous process, it yields
\begin{align}
h\sum_{j=1}^M\max_{1\leq m\leq M}|\Delta_xu_{mj}|^3\leq \|\Delta_x u\|_4^2\braB{4\|\delta_x\delta_yu\|+\frac{1}{L}\|\Delta_x u\|}.\label{ineq:bound3}
\end{align}

We now estimate $\|\Delta_xu\|_6$. Using Cauchy-Schwarz inequality, we have
\begin{align*}
&h^2\sum_{i=1}^M\sum_{j=1}^M|\Delta_xu_{ij}|^6
=h\sum_{i=1}^M\braB{h\sum_{j=1}^M|\Delta_xu_{ij}|^3\cdot|\Delta_xu_{ij}|^3}\\
\le&h\sum_{i=1}^M\braB{\max_{1\le j\le M}|\Delta_xu_{ij}|^3\cdot h\sum_{j=1}^M|\Delta_xu_{ij}|^3}\\
\le&\braB{h\sum_{i=1}^M\max_{1\le j\le M}|\Delta_xu_{ij}|^3}\braB{h\sum_{j=1}^M\max_{1\le i\le M}|\Delta_xu_{ij}|^3}
\end{align*}
Substituting (\ref{ineq:bound2}) and (\ref{ineq:bound3}) into above inequality, it leads to
\begin{align*}
\|\Delta_xu\|_6^6&\le \|\Delta_x u\|_4^4\braB{4\|\delta_x^2u\|+\frac{1}{L}\|\Delta_x u\|}\braB{4\|\delta_x\delta_yu\|+\frac{1}{L}\|\Delta_x u\|}\\
&\le 5\|\nabla_h u\|_4^4\braB{4\|\Delta_hu\|^2+\frac{1}{L^2}\|\nabla_h u\|^2}.
\end{align*}
Similarly, we obtain $$\|\Delta_yu\|_6^6\le5 \|\nabla_h u\|_4^4\braB{4\|\Delta_hu\|^2+\frac{1}{L^2}\|\nabla_h u\|^2}.$$
Then it follows that
\begin{align*}
\|\nabla_hu\|_6^6&=h^2\sum_{i=1}^M\sum_{j=1}^M\braB{|\Delta_xu_{ij}|^2+|\Delta_yu_{ij}|^2}^3\\&
\le 4h^2\sum_{i=1}^M\sum_{j=1}^M\braB{|\Delta_xu_{ij}|^6+|\Delta_yu_{ij}|^6}\\
&\le 40\|\nabla_h u\|_4^4\braB{4\|\Delta_hu\|^2+\frac{1}{L^2}\|\nabla_h u\|^2}.
\end{align*}
Following the same procedure of the proof for $\|\nabla_hu\|_6$, we obtain the estimate of $\|\nabla_hu\|_4$
$$\|\nabla_hu\|_4\le C_2 \|\nabla_hu\|^\frac{1}{2}\Big(2\|\Delta_hu\|^2+\frac{1}{L^2}\|\nabla_hu\|^2\Big)^\frac{1}{4},$$
where $C_2$ is a constant.
 The above two inequalities imply that there exists a constant $K$ such that
\begin{align*}
\|\nabla_hu\|_6\le K\|\nabla_hu\|^{\frac{1}{3}}\braB{4\|\Delta_hu\|^2+\frac{1}{L^2}\|\nabla_hu\|^2}^{\frac{1}{3}}.
\end{align*}
}


\end{document}